\documentclass[10pt,a4paper]{article}
\usepackage[left=1.75cm,right=1.75cm,top=2cm,bottom=2cm]{geometry} 

\usepackage{subcaption}
\usepackage{amssymb}
\usepackage{amsmath}
\usepackage{amsthm}
\usepackage{xcolor}
\usepackage{graphicx}
\usepackage{epstopdf}
\usepackage{siunitx}
\usepackage{enumerate}
\usepackage[]{algorithm2e}
\usepackage{comment}
\usepackage{psfrag}
\usepackage[breaklinks,bookmarks=false]{hyperref}
\usepackage[textsize=tiny,color=yellow!60!green]{todonotes}
\usepackage{graphicx}
\usepackage{color}
\usepackage{mathbbol}
\usepackage{amssymb}   
\DeclareSymbolFontAlphabet{\amsmathbb}{AMSb}%
\usepackage{float}
\usepackage{array}
\usepackage{multirow}
\usepackage{dsfont}
\usepackage{mathrsfs}
\usepackage[numbers]{natbib}
\usepackage{amsfonts}
\usepackage{amssymb}
\usepackage{enumitem}
\newcommand{\bse}{\begin{subequations}}
\newcommand{\ese}{\end{subequations}}

\newcommand{\pt}{\partial_t}

\newcommand{\Assum}{(\textbf{A1})--(\textbf{A6})}

\newcommand{\bid}{\mathbf{I}}
\newcommand{\bK}{\mathbf{K}}

\newcommand{\bth}{\boldsymbol{\Theta}}

\newcommand{\bu}{\mathbf{u}}
\newcommand{\btheta}{\boldsymbol{\theta}}
\newcommand{\by}{\mathbf{y}}
\newcommand{\bv}{\mathbf{v}}
\newcommand{\bff}{\mathbf{f}}
\newcommand{\bw}{\mathbf{w}}
\newcommand{\bz}{\mathbf{z}}
\newcommand{\br}{\mathbf{r}}
\newcommand{\be}{\mathbf{e}}

\newcommand{\bep}{\boldsymbol{\varepsilon}}

\DeclareMathOperator{\divr}{div}

\newcommand{\Wcal}{\mathcal{W}}
\newcommand{\Ucal}{\mathcal{U}}
\newcommand{\Pcal}{\mathcal{P}}
\newcommand{\Tcal}{\mathcal{T}}
\newcommand{\Rcal}{\mathcal{R}}

\newcommand{\Xcal}{\mathcal{X}}

\newcommand{\Mcal}{\mathcal{M}}

\newcommand{\real}{{\amsmathbb R}}

\newcommand{\rtzero}{\amsmathbb{RT}_0}
\newcommand{\pspace}{\amsmathbb{P}}

\newcommand{\norm}[1]{\left\lVert#1\right\rVert}

\newtheorem{theorem}{Theorem}[section]
\newtheorem{corollary}{Corollary}[theorem]
\newtheorem{lemma}[theorem]{Lemma}
\newtheorem{remark}{Remark}[section]

\newtheorem{df}[theorem]{Definition}

\let\phi\varphi

\makeatletter
\newcommand{\pushright}[1]{\ifmeasuring@#1\else\omit\hfill$\displaystyle#1$\fi\ignorespaces}
\newcommand{\pushleft}[1]{\ifmeasuring@#1\else\omit$\displaystyle#1$\hfill\fi\ignorespaces}
\makeatother

\begin{document}

\title{Monolothic and splitting solution schemes for  fully coupled quasi-static thermo-poroelasticity with nonlinear 
convective transport\thanks{This work forms part of Norwegian Research Council project 250223}}

\author{Mats Kirkes\ae{}ther Brun\footnotemark[2]\and
Elyes Ahmed\footnotemark[2]
\and Inga Berre\footnotemark[2]
\and Jan Martin Nordbotten\footnotemark[2]\ \footnotemark[3]
\and Florin Adrian Radu\footnotemark[2] 
}
\date{\today}
\maketitle

\renewcommand{\thefootnote}{\fnsymbol{footnote}}

\footnotetext[2]{Department of Mathematics, University of Bergen, P. O. Box 7800, N-5020 Bergen, Norway.
\href{mailto:mats.brun@uib.no}{mats.brun@uib.no},
\href{mailto:elyes.ahmed@uib.no}{elyes.ahmed@uib.no},
\href{mailto:inga.berre@uib.no}{inga.berre@uib.no},
\href{mailto:jan.nordbotten@uib.no}{jan.nordbotten@uib.no},
\href{mailto:florin.radu@math.uib.no}{florin.radu@uib.no}.
}
\footnotetext[3]{Department of Civil and Environmental Engineering, Princeton University, Princeton, N. J., USA.}
\renewcommand{\thefootnote}{\arabic{footnote}}

\numberwithin{equation}{section}
    \setcounter{secnumdepth}{5}
    \renewcommand{\thesubsubsection}{\thesubsection.\Alph{subsubsection}}
\begin{abstract}
This paper concerns splitting-based iterative procedures for the coupled nonlinear thermo-poroelasticity model problem. The thermo-poroelastic model problem we consider is formulated as a three-field system of PDE's, consisting of an energy balance equation, a mass balance equation and a momentum balance equation, where the primary variables are temperature, fluid pressure, and elastic displacement. Due to the presence of a nonlinear convective transport term in the energy balance equation, it is convenient to have access to both the pressure and temperature gradients. Hence, we introduce these as two additional variables and extend the original three-field model to a five-field model. For the numerical solution of this five-field formulation, we compare  three   approaches that differ by how we treat the coupling/decoupling   between the flow and/from heat and/from mechanics; these approaches have in common a simultaneous application of the fixed-stress splitting scheme  on both the non-linearity and the    coupling structure of the problem. More precisely, the derived procedures   transform a nonlinear and fully coupled problem into a set of simpler subproblems to be solved sequentially in an iterative fashion.   We  provide a convergence proof for the derived algorithms, and validate our results through several numerical examples.
\end{abstract}
\vspace{3mm}

\noindent{\bf Key words:} Quasi-static thermo-poroelasticity; nonlinear convective transport; porous media; monolithic scheme; fixed-stress splitting iterative coupling;  L-scheme linearization;  contraction mapping; mixed finite elements.


\pagestyle{myheadings} \thispagestyle{plain} \markboth{M. K. Brun}{}
\section{Introduction}\label{sec:presentation}
\subsection{Problem statement}
The field of \emph{poroelasticity} is concerned with the interaction between viscous fluid flow and elastic solid deformation within a porous material, and was pioneered through the works of K. Terzhagi~\cite{terzaghi1944theoretical} and M. A. Biot~\cite{biot1941general, biot1972theory}. In the fully-saturated, quasi-static regime, the mathematical modeling of such processes constitutes a coupled two-field linear model where the primary variables are the fluid pressure and the elastic displacement of the solid. This is known as the quasi-static Biot's model. 

In many important applications, such 
as geothermal energy extraction, nuclear waste disposal and carbon storage, 
temperature also plays a vital role and must therefore 
be included in the aforementioned model. Thus, we consider here a \emph{thermo-poroelastic} system which can be seen as a generalization of the Biot system to the non-isothermal case, i.e., the coupled processes are heat, flow, and geomechanics. Since it appears to be the cornerstone of many complex models, we focus on the following nonlinear and coupled quasi-static thermo-poroelastic equations as it is exposed in~\cite{brun2018thporo}: find 
the temperature $T$, the pressure $p$, and the displacement 
$\bu$ such that
\bse\label{thporo}
	\begin{alignat}{3}
	\partial_{t}\psi(p,\bu,T) + c_f(\bK \nabla p) \cdot \nabla T 
		- \nabla \cdot(\bth \nabla T) &=z,\quad&&\textnormal{in } \Omega \times (0,t_{f}),
		\label{thporo:primalheat}\\
		-\nabla \cdot \btheta(\bu) + \alpha \nabla p + \beta \nabla T &= \bff,
		\quad&&\textnormal{in }\Omega \times (0,t_{f}),\label{thporo:primalflow}\\
		\partial_t \phi(p,T,\bu) - \nabla \cdot(\bK \nabla p) &= g,\quad&&\textnormal{in } \Omega \times (0,t_{f}),\label{thporo:primalmech}\\
		\label{thporo:BC}T(\cdot,0) = T_0, \quad \bu(\cdot,0) = \bu_0,\quad p(\cdot,0) &= p_0, \quad&& \textnormal{in } \Omega,\\
		\label{thporo:IC}T = 0, \quad \bu = 0, \quad p &= 0, \quad &&\textnormal{on } \partial\Omega \times (0,t_{f}).
	\end{alignat}
\ese
In the above model,   $\Omega$ is a bounded (connected and open) domain in $\real^{d}$, $d=2$ or $3$,  and  $t_f > 0$ is the final time. The function $z$ is the heat source, $g$ is the mass source, and $\bff$ is the  body force.  The functionals $\psi$ and $\phi$ denote the heat content and fluid content, respectively, i.e.,  $\psi(p,\bu,T):=a_0T - b_0p + \beta \nabla \cdot \bu$, and $\phi(p,\bu,T):=c_0p - b_0 T + \alpha \nabla \cdot \bu$,  where  $c_{0}$   is the constrained-specific storage coefficient, $a_0$ is the effective thermal capacity, $b_0$ is the thermal dilation coefficient,  $\alpha$  is the Biot--Willis constant, and $\beta $ is the thermal stress coefficient.  

The parameter $c_{f}$ is the volumetric heat capacity of the fluid,  $\bK = (K_{ij})_{i,j=1}^d$ is the permeability divided by fluid viscosity, and $\bth = (\Theta_{ij})_{i,j=1}^d$ is the effective thermal conductivity. The function $\btheta$ denotes the effective stress tensor, i.e.,  $\btheta(\bu):= 2\mu \bep(\bu) + \lambda \nabla \cdot \bu \bid $, where $\bep(\bu) := (\nabla \bu + \nabla \bu^\top)/2$ the symmetric part of $\nabla \bu$, and  $\bid$  is the identity tensor. Finally,   $T_0$  is the initial temperature, $\bu_0$ is the initial displacement and $p_0$ is the initial pressure.

Note that the above model introduces a nonlinearity with a coupling term being the convective transport  
term in the energy balance~\eqref{thporo:primalheat}, which strongly complicates the
problem compared to the isothermal case (i.e., the linear Biot's model). Note that if $b_0 = \beta = 0$, the flow and mechanics decouples from the heat, and Biot's model is recovered. For the  derivation of the constitutive equations of thermo-poroelasticity we refer to the works~\cite{van2017thermoporoelasticity, suvorov2011macroscopic, gatmiri1997formulation}, and particularly to~\cite{brun2018thporo, van2017thermoporoelasticity} where the above model was derived within the framework of the two-scale asymptotic expansion method (see \textit{e.g.}~\cite{hornung2012homogenization} for a review of this technique).
\subsection{Weak solution and well-posedness of the continuous problem}\label{subsec:var}
The common structure of mathematical models which are based
on (systems of) scalar conservation laws of the form~\eqref{thporo:primalheat} and where nonlinear gradient terms appear, suggests  introducing  the heat flux; $\br := -\bth \nabla T$, or the Darcy flux; $\bw := -\bK \nabla p$, as an additional variable, thus, either the nonlinear coupling term $c_f(\bK \nabla p) \cdot \nabla T$ becomes  $\left[-c_f(\bw\cdot \nabla T)\right]$ or $\left[-c_f((\bK\otimes \bth^{-1})\br\cdot \nabla p)\right]$, \textit{e.g.}~\cite{MR2116915, MR2039576}. Precisely, it is well known that this term,  dealing non-linearly  with the coupling convective term, can be quite difficult to approximate correctly in its actual form. This altogether leads to challenging numerical issues. Furthermore,  the choice to introduce the heat flux  or the Darcy flux as a new variable   depends strongly on which process (the flow or the heat flow) dominates, and may  result   in a different treatment of the convective term. Here,  to avoid some of these complexities, we adopt from~\cite{brun2019well}  the mixed form for both the heat and flow subproblems~\eqref{thporo:primalheat} and \eqref{thporo:primalflow}, taking in mind that Mixed Finite Element (also Finite Volume) literature  has developed techniques to handle convective terms~\cite{MR2846759,MR2346377}.  Throughout the paper, we assume that the following 
 assumptions hold true
 \begin{enumerate}[label=(\textbf{A\arabic*})]\label{assumpA}
\item $\bK: \real^{d}\rightarrow \real^{d\times d}$ is assumed  to be constant in time, symmetric, definite and positive; there exist $k_m>0$ and $k_M$ such that
$k_m |\zeta|^2 \leq \zeta^\top \bK(x) \zeta \textnormal{ and } |\bK(x)\zeta| \leq k_M |\zeta|, \ \forall \zeta \in \real^d \setminus \{0\}.$\label{Kdef}
\item $\bth: \real^{d}\rightarrow \real^{d\times d}$ is assumed  to be constant in time, 
symmetric, definite and positive; there exist $\theta_m>0$ and $\theta_M$ such that $\theta_m |\zeta|^2 \leq \zeta^\top \bth(x) \zeta \textnormal{ and } |\bth(x)\zeta| \leq \theta_M |\zeta|, \ \forall \zeta \in \real^d \setminus \{0\}$.\label{thdef}
\item The coefficients $a_{0}$, $b_{0}$, $c_{0}$, $c_{f}$, $\alpha$ and $\beta$ are strictly positive constants.
\item The coefficients  $a_0,b_0$ and $c_0$ are such that 
$c_0-b_0 >0$ and $a_0-b_0 > 0$. \label{constraints}
\item The source terms are such that  
$z, g\in L^2(0,t_f;L^2(\Omega))$ and  $\bff\in H^1(0,t_f;L^2(\Omega))$.
\item The initial data are such that   $p_0, T_0 \in H^1_0(\Omega)$ and $\bu_0 \in (L^2(\Omega))^d$.
\end{enumerate}
Before transcribing the mixed variational formulation of the problem, we introduce some  notations: $$\Tcal := L^2(\Omega),\ \ \Rcal := H(\divr,\Omega),\ \ \Pcal := L^2(\Omega),\ \ \Wcal := H(\divr,\Omega), \ \ \Ucal := (L^2(\Omega))^d,$$ where we denote by $(\cdot, \cdot)$ the standard $L^2(\Omega)$ inner product, and by $\norm{\cdot}$ the induced $L^2(\Omega)$ norm. Due to \ref{Kdef} and \ref{thdef}, we can use the tensors $\bK$ and $\bth$ (and their inverses) to define $L^2(\Omega)$-equivalent norms which we denote by $\norm{\bu}_{\bK} := (\bK \bu, \bu)^{1/2}$ (and $\norm{\bu}_{\bK^{-1}} := (\bK^{-1}\bu, \bu)^{1/2}$), and similarly for $\bth$. The  variational formulation of~\eqref{thporo} then reads as follows:
\begin{df}[The continuous formulation~\cite{brun2019well}] Assuming that~\textnormal{\Assum}  hold true.  The fully  coupled mixed-primal formulation of~\eqref{thporo} reads as follows:  find $(T(t), \br(t), p(t), \bw(t), \bu(t)) \in \Tcal \times \Rcal \times \Pcal \times \Wcal \times \Ucal,$ such that for a.e. $t \in (0,t_f)$ there holds
\bse\label{nonlinvar}
	\begin{alignat}{2}
		(\partial_{t}\psi(p,T,\bu), S) 
		+ c_f(\bw \cdot \bth^{-1} \br, S) + (\nabla \cdot \br,S) &= (z,S), \quad 
		&&\forall S \in \Tcal,\label{vf1}\\
		(\bth^{-1} \br, \by) - (T, \nabla \cdot \by) &= 0, \quad
		&&\forall \by \in \Rcal, \label{vf2}\\
		(\pt \phi(p,T,\bu), q) 
		+ (\nabla \cdot \bw,q) &= (g,q), \quad
		&&\forall q \in \Pcal, \label{vf3}\\
		(\bK^{-1} \bw, \bz) - (p,\nabla \cdot \bz) &= 0, \quad 
		&&\forall \bz \in \Wcal, \label{vf4}\\
		(\btheta(\bu),\bep(\bv)) -(\beta T+ \alpha p, \nabla \cdot \bv) &= (\bff,\bv), \quad
		&&\forall \bv \in \Ucal, \label{vf5}
	\end{alignat}
 together with the initial condition~\eqref{thporo:IC}. 
\ese
\end{df}
The above variational problem was analyzed in \cite{brun2019well}. There, it was shown that under the assumption that the heat flux (or Darcy flux) is such that $\br(t) \in (L^\infty(\Omega))^d$, for $t \in (0,t_f)$, the problem~\eqref{nonlinvar} has a unique weak solution. Moreover, it was shown that with additional regularity on the data, i.e., $\bff \in H^2\left(0,t_f;(L^2(\Omega))^d\right)$, $h,g \in H^1(0,t_f;L^2(\Omega))$, and $T_0, p_0 \in H_0^1(\Omega) \cap H^2(\Omega)$, the fluxes are bounded functions.
\subsection{Goal and positioning of the paper}
The simulation of thermo-poroelasticity problems  is very difficult because of the coexistence of different physics  that require coupling. For this type of problems, there are typically three different coupling approaches employed in modeling fluid flow coupled with reservoir geomechanics. They are known as the fully implicit, the explicit (loosely or weakly) coupling, and the \textit{splitting-iterative} approaches. The main
issue for the applicability of the implicit approach that  solves  
simultaneously the above \textit{three-phenomena} (flow, heat and  mechanics) problem is that  it results in a very large system to be
solved at each time step.  Note however that  this approach has  excellent stability properties~\cite{MR3899054,MR1690489}. If going weakly coupled, the resulting  approach enjoys a lower computational cost compared to the implicit (monolithic) one but known to be inaccurate in general, and only conditionally stable~\cite{MR2284413,MINKOFF200337}. Here, we adopt the iterative coupling approach that lies  between the implicit and explicit approach,  and resolves iteratively the two/three subsystems (depending on the choice of decoupling) by exchanging the values of the shared state variables in an iterative fashion~\cite{both2017robust,mikelic2014numerical}.

In this paper, adopting an iterative method for  a \textit{nonlinear and fully coupled three-processes} problem    appears    as a natural approach (more than an option!) as the implicit one leads  to  a huge system (particularly if MFE methods are adopted~\cite{ahmed:hal-01687026,castelletto2014unified,garcia2005numerical,MR3473457})  incorporating different equations varied in type,  such as coupling linear and non-linear systems, that must  be solved at each time step.  The advantage of the iterative approaches transcribed in this paper is that, at each iteration, smaller, easier-to-solve systems cooperate iteratively through algorithms~\cite{garcia2005numerical,MR3343599}.  Another advantage that distinguishes our approaches is  the possibility to \textit{reuse existing   codes} for different numerical schemes and coupling techniques specialized to each component of the problem~(see~\cite{MR3624734,MR2446503} for multiples processes model). For the classical linear  poroelasticity, the iterative coupling procedures mentioned in the above has been studied extensively~\cite{ahmed2019adaptive,both2017robust, castelletto2015accuracy, iliev2016numerical, kim2009stability, kolesov2017splitting, mikelic2014numerical, mikelic2013convergence, tran2005overview}. In particular, two such algorithms have received considerable attention; the \textquotedblleft Undrained Split\textquotedblright (constant fluid mass during structure deformation),  and the \textquotedblleft Fixed Stress Split\textquotedblright (constant volumetric mean total stress during solution of flow problem). In~\cite{kim2009stability} these were first shown to be unconditionally stable. In~\cite{mikelic2014numerical, mikelic2013convergence} contraction estimates and rates of convergence were derived. 

Building further on the idea of the Undrained Split/Fixed Stress Split algorithms is the \textit{$L$-scheme}. For solving of coupled problems, this involves adding an artificial stabilization term to one or more of the subproblems with a parameter $L>0$. Here, the quantity held constant during solving of one of the subproblems needs not have any physical interpretation. In this sense, the $L$-scheme generalizes the Undrained Split/Fixed Stress Split algorithms, and due to the removal of physical constraints on the stabilization terms, allows for further optimization. The $L$-scheme can also be employed as a linearization procedure for nonlinear problems, with the parameter $L>0$ mimicking the Jacobian from Newton iteration. However, in order to determine the parameter $L>0$, for any given problem, derived convergence estimates are necessary. The $L$-scheme has been shown to perform robustly for Richards equation~\cite{list2016study, MR2079503} and for both linear and nonlinear coupled flow and geomechanics~\cite{MR3827264,both2017robust}. 

Although the literature on iterative coupling procedures for (isothermal) poroelastic problems is quite extensive, thermo-poroelastic problems have not received the same amount of attention. Sequential iterative methods for thermo-poroelasticity was considered in~\cite{kim2015unconditionally}, but for the linear case. Iterative splitting schemes for separate poroelasticity and thermoelasticity  problems were considered in~\cite{kolesov2014splitting}. Compared to problems of (two-field) coupled flow and mechanics (which can be solved either sequentially or monolithically) we now have additional options in partial decoupling, i.e., solving two of the subproblems together decoupled from the third. Combinatorially, this yields six combinations of iterative procedures, ranging from monolithic to fully decoupled. In this work, we propose six iterative algorithms for thermo-poroelasticity based on these six combinations of coupling/decoupling. In particular, we employ variations of the $L$-scheme in all six algorithms, with artificial stabilization terms added to both the flow and heat sub-problems. Here, the main \textit{advantage} of the $L$-scheme is that it treats simultaneously the coupling and the non-linearity effects, thus,  no inner iterative approaches are required, see \textit{e.g.}~\cite{MR3771899} where $L$-scheme type  approaches  are developed to treat iteratively
a combined domain decomposition and nonlinearity problem.   The convergence in most cases is  linear but in the required energy norms. Furthemore, the necessary constraint on the time step is not severe.

The reason we propose all the possible algorithms is the following: the coupling strength of the heat, flow and mechanics may vary depending on the physics at hand. Precisely, to develop robust and efficient solution procedures for the   \textit{three-processes problem} at hand, one  should take into account
which process (\textit{the mechanics} and/or  \textit{flow} and/or  \textit{heat flow}) dominates the full
problem. Thus, to be   agnostic  towards  
the dominating processes and derive  a 
framework for this model problem, we propose six variations of iterative coupling/decoupling algorithms for thermo-poroelasticity, covering all possibilities of varying coupling strength between the three physical processes involved. This derivation  is also important for  practitioners  to \textit{reuse existing   codes} where  smaller problems and/or algorithms  can be  \textit{easily coded,  cheaply combined},  and
\textit{efficiently used} in  \textit{practical simulations}. Note that developed algorithms are applicable on any numerical schemes used to obtain the solutions of the different  processes~\cite{phillips2008coupling,yi2014convergence}.  For the convergence analysis,  we derive energy-type estimates, from which we infer the convergence of the iterate solutions as well as obtaining strict lower bounds on the stabilization parameters, and an upper bound on the time step. However,  a \textquotedblright cut-off\textquotedblleft  \ operator $\Mcal$ is introduced in the mixed setting in order to make the iterative schemes
converge. Several numerical tests validate our proposed algorithms. In particular, we show that by using the derived stabilization estimates, the proposed algorithms perform robustly with respect to both mesh refinement and a wide range of different problem parameters.

The article is organized as follows: In Section~\ref{sec:discr} we present the fully discretization of the  thermo-poroelasticity model, and in Section~\ref{sec:algos} we present all six iterative algorithms. In Section~\ref{sec:conv}, convergence analysis based on contraction estimates are derived, from which the well-posedness of the discrete scheme is inferred 
in addition to the bounds on the stabilization parameters and time step. In Section~\ref{sec:experiments} we provide several numerical experiments, and finally in Section~\ref{sec:conclusions} some concluding remarks.
\section{Discrete setting}\label{sec:discr}
 Let  $\Xcal_h$ be a simplicial  mesh of $\Omega$,  
matching in the sense that for two distinct elements of $\Xcal_h$ their
intersection is either an empty set or their common vertex or edge.   Let $h_{K}$ denote the diameter of $K\in \Xcal_h$ and let $h$
be the largest diameter of all such triangles, i.e., $h:= \max_{K\in \Xcal_{h}}h_K$. For the time partition, we let $\{ t^n : n=0,1,\cdots,N \}$ be the discrete time steps, where $0 := t^0 < t^1 < \cdots < t^N = t_f$, and  
$\tau^n = t^n - t^{n-1}$, $n\geq 1$, be  the difference between consecutive discrete times. In other words, we have   $t^{n}:=\sum_{\ell=1}^{n}\tau^{\ell}, \ 1 \leq n\leq  N$, and therefrom $t_f=\sum_{n=1}^{ N}\tau^{n}$.

For the discrete spaces, we let $\Tcal_h, \Rcal_h, \Pcal_h, \Wcal_h$ and $\Ucal_h$ be suitable finite element spaces corresponding to the infinite dimensional spaces of subsection~\ref{subsec:var}, where we assume that 
\begin{equation}\label{divcondition}
\divr \Rcal_h = \Tcal_h \quad \textnormal{ and } \quad \divr \Wcal_h = \Pcal_h. 
\end{equation}
For the time discretization we will employ a backward Euler scheme. 
For the sake of simplicity, we assumed the source terms $\bff$, $g$ and $z$ to be piecewise constant in time. We then denote by $(T_h^n, \br_h^n, p_h^n, \bw_h^n, \bu_h^n)$ the discrete counterpart of the solution tuple to problem \eqref{nonlinvar} at time $t^n$. 
\begin{df}[The  coupled $\textit{mixed}\times\textit{mixed}$ and \textit{Galerkin} finite element scheme]\label{def:discretevar}
The discrete formulation of the problem \eqref{nonlinvar}  reads: given 
$\psi(p^{0}_{h},T^{0}_{h},\bu_{h}^{0})$ and $\phi(p^{0}_{h},T^{0}_{h},\bu_{h}^{0})$, then, for $n=1,\cdots,N$,  find $(T^n_h, \br^n_h, p^n_h, \bw^n_h, \bu^n_h) \in \Tcal_h \times \Rcal_h \times \Pcal_h \times \Wcal_h \times \Ucal_h$ such that,
\bse\label{discretevar}
	\begin{alignat}{2}
		\nonumber &(\psi(p^{n}_{h},T^{n}_{h},\bu_{h}^{n}),S_h)
		+ \tau^n c_f(\bw^{n,M}_h \cdot \bth^{-1} \br^{n,M}_h, S_h) + \tau^n(\nabla \cdot \br^n_h,S_h) \\
		&\quad\qquad \quad\quad\qquad \quad\quad\qquad \qquad= \tau^n(z^{n},S_h) + (\psi(p^{n-1}_{h},T^{n-1}_{h},\bu_{h}^{n-1}),S_h), \quad 
		&&\forall S_h \in \Tcal_h,\label{discr1} \\
		&(\bth^{-1} \br^n_h, \by_h) - (T^n_h, \nabla \cdot \by_h) = 0, \quad
		&&\forall \by_h \in \Rcal_h, \label{discr2}\\
		 &\psi(p^{n}_{h},T^{n}_{h},\bu_{h}^{n}), q_h)
		+ \tau^n(\nabla \cdot \bw_h^n,q_h) = \tau^n(g^{n},q_h) + (\psi(p^{n-1}_{h},T^{n-1}_{h},\bu_{h}^{n-1}), q_h), \quad
		&&\forall q_h \in \Pcal_h, \label{discr3} \\
		&(\bK^{-1} \bw^n_h, \bz_h) - (p^n_h,\nabla \cdot \bz_h) = 0, \quad 
		&&\forall \bz_h \in \Wcal_h, \label{discr4}\\
		&2\mu(\bep(\bu^n_h),\bep(\bv_h)) + \lambda (\nabla \cdot \bu^n_h, \nabla \cdot \bv_h) 
		- (\beta T^n_h +\alpha p^n_h, \nabla \cdot \bv_h) = (\bff^{n},\bv_h), \quad
		&&\forall \bv_h \in \Ucal_h. \label{discr5}
	\end{alignat}
\ese
where the functions $(\bw^{n,M}_h,\br^{n,M}_h)$ are defined as,
\begin{alignat}{2}\label{convective_term_approx}
       \bw^{n,M}_h:=\min(|\bw^{n}_h|,M)\dfrac{\bw^{n}_h}{|\bw^{n}_h|},\quad\textnormal{and }\quad\br^{n,M}_h:=\min(|\br^{n}_h|,M)\dfrac{\br^{n}_h}{|\br^{n}_h|},
      \end{alignat}
      with $M$ is a  fixed positive real number and $|\bv|:=\sqrt{\sum_{i=1}^{d}(\bv)_{i}^{2}}$.
\end{df}
In the above scheme,  we used $(\bw^{n,M}_h \cdot \bth^{-1} \br^{n,M}_h, S_h)$ for  the approximation of the convective coupling term    
instead of the original $(\bw^{n}_h \cdot \bth^{-1} \br^{n}_h, S_h)$. The reason for this  approximation 
 will be clarified later. The  equations~\eqref{discr1}-\eqref{discr2} form the discrete mixed scheme of the   \textit{heat subproblem},   \eqref{discr3}-\eqref{discr4} form the  discrete mixed scheme  for the \textit{flow subproblem}, and  \eqref{discr5} is  the discrete form of  the \textit{mechanics subproblem} with Galerkin finite element method.  Together, these subproblems make up the  nonlinear and fully coupled  discrete version of the \textit{thermo-poroelastic problem} to be solved iteratively in the next section. 
\begin{remark}[Other schemes]
 The results we present  are valid also for other choices of temporal discretizations, as well as different (i.e., non-mixed) formulations for the heat and flow problems. Different spatial discretizations can even be chosen for each of the three subproblems, although we do not pursue this topic further.
 \end{remark}
 \begin{remark}[Convective coupling term]
 The convective coupling term $(\bw^{n}_h \cdot \bth^{-1} \br^{n}_h, S_h)$ can also be approximated by $(\bw^{n,M}_h \cdot \bth^{-1} \br^{n,R}_h, S_h)$, where two different constants $M$ and $R$ are used  in the definitions~\eqref{convective_term_approx}. In that case, the underlying iterative methods of Section~\ref{sec:algos} as well as the convergence analysis of Section~\ref{sec:conv} remains true with minor modifications in the proofs.
 \end{remark}
\section{The   $L$-type iterative schemes}\label{sec:algos}
We now present six iterative (splitting) algorithms for the discrete thermo-poroelastic problem \eqref{discretevar}. These algorithms involve either decoupling all the subproblems and solving each separately at every iteration (three-step algorithm), or decoupling only one subproblem from the other two which are then solved together (two-step algorithm), or solving a linearized problem monolithically at every iteration (one-step algorithm). We use the letters \textbf{H} (Heat), \textbf{F} (Flow), and \textbf{M} (Mechanics), to abbreviate the algorithms, \textit{e.g.} a two-step algorithm where the heat and flow subproblems are solved together decoupled from the mechanics subproblem is referred to as (\textbf{HF-M}), and similarly for other combinations of coupling/decoupling of the subproblems. Throughout the rest of the article we will mostly refer to the discrete problems, and therefore omit the $h$-subscript on the variables and test functions for cleaner notation. We shall also denote the time step simply by $\tau$, keeping in mind it may depend on $n$.

At the time step $n \geq 1$, let $(T^{n-1}, \br^{n-1}, p^{n-1}, \bw^{n-1}, \bu^{n-1})$ be given. We then approximate the solution  at the actual time step $n\in\{1,\cdots,N\}$, using the  sequence  $(T^{n,i}, \br^{n,i}, p^{n,i}, \bw^{n,i},\bu^{n,i})$ for $i \geq 0$, defined in an iterative fashion,  and where the iterate   $(T^{n,0}, \br^{n,0}, p^{n,0}, \bw^{n,0},\bu^{n,0})$ is an initial  guess. All the algorithms involve adding the stabilization terms $L_T(T^{n,i} - T^{n,i-1}, S)$ and $L_p(p^{n,i} - p^{n,i-1},q)$ to the left hand sides of equations \eqref{discr1} and \eqref{discr3}, respectively, where $L_T,L_p > 0$ are the stabilization parameters (to be chosen later). Furthemore, to make the notation easier, we introduce the parametrized fluid and heat content functionals: for a given $L_T,L_p > 0$, we define
\bse\begin{alignat}{2}
& \psi_{L_{T}}(p,\bu,T):=(a_0+ L_{p})T - b_0p + \beta \nabla \cdot \bu,\\
&\phi_{L_{p}}(p,\bu,T):=(c_0+ L_{T})p - b_0T + \alpha \nabla \cdot \bu.
\end{alignat}\ese

For the analysis of the coupled mixed~\eqref{discretevar} and  the underlying iterative approach introduced in this section, we need to introduce  the cut-off operator $\Mcal$ as described  in \textit{e.g.}~\cite{MR2116915, MR2039576} by
\begin{equation}\label{cutoffoperator}
\Mcal(\bz)(x) := 
\begin{cases}
\bz(x), &|\bz(x)| \leq M,\\
M\bz(x) / |\bz(x)|, &|\bz(x)| > M,
\end{cases}
\end{equation}
where $M$ is a large positive constant. The notation $(\bw^{n,M}_h,\br^{n,R}_h)$ used in Definition~\ref{def:discretevar} is then equivalent to  $\bw^{n,M}_h$. Note that the use of $(\bw^{n,M}_h,\br^{n,R}_h)$ instead of $(\bw^{n}_h,\br^{n}_h)$  has little or no practical implications, but is necessary in order to facilitate the convergence analysis; obviously,  if the exact fluxes are bounded, i.e.,  ${\bw}^n,{\br}^n\,\in (L^\infty(\Omega))^d$, then  if we picked $M$ large enough, we have practically  $\Mcal({\bw}^n)(x) = {\bw}^n(x)$ and $\Mcal({\br}^n)(x) = {\br}^n(x)$.
The iterative algorithms then reads as follows:
\subsection{The  monolithic scheme (HFM)}\label{sect:algo1}
At the each iteration $i>0$ of the $L$-type monolothic scheme, we solve the  linearized thermo-poroelastic problem: given $(T^{n,i-1}, p^{n,i-1}, \bw^{n,i-1}, \bu^{n,i-1})$,  find $(T^{n,i}, \br^{n,i}, p^{n,i}, \bw^{n,i}, \bu^{n,i})$ such that
\bse\label{Lscheme_mono}
	\begin{alignat}{2}
		\nonumber &(\psi_{L_{T}}(T^{n,i},p^{n,i},\bu^{n,i}), S)
		+ \tau c_{f}(\Mcal(\bw^{n,i-1}) \cdot \bth^{-1}\Mcal( \br^{n,i}), S)  \\
		 &\qquad \qquad\qquad  + \tau(\nabla \cdot \br^{n,i},S)= \tau(z^{n},S) + (\psi(T^{n-1},p^{n-1},\bu^{n-1}), S)+L_T(T^{n,i-1},S), \qquad 
		&&\forall S \in \Tcal_h, \\
		&(\bth^{-1} \br^{n,i}, \by) - (T^{n,i}, \nabla \cdot \by) = 0, \quad
		&&\forall \by \in \Rcal_h, \\
		 \nonumber&(\phi_{L_{p}}(T^{n,i},p^{n,i},\bu^{n,i}), q)
		+ \tau(\nabla \cdot \bw^{n,i},q) \\
		 &\qquad \qquad = \tau (g^{n},q) + 
		(\phi(T^{n-1},p^{n-1},\bu^{n-1}), q) + L_{p}(p^{n,i-1},q), \hfill
		&&\forall q \in \Pcal_h, \\
		&(\bK^{-1} \bw^{n,i}, \bz) - (p^{n,i},\nabla \cdot \bz) = 0, \hfill 
		&&\forall \bz \in \Wcal_h,\\
		&2\mu(\bep(\bu^{n,i}),\bep(\bv)) + \lambda (\nabla \cdot  \bu^{n,i}, \nabla \cdot \bv) -  (\beta T^{n,i} +\alpha p^{n,i}, \nabla \cdot \bv)= (\bff^{n},\bv), 
		&&\forall \bv \in \Ucal_h.
	\end{alignat}
\ese
This algorithm in continued until a fixed tolerance is reached. Clearly, in the  above algorithm,  the $L$-scheme acts only as a linearization procedure, where we approximate
the  convective transport term by $\Mcal(\bw^{n,i-1}) \cdot \bth^{-1}\Mcal( \br^{n,i})$. Note that, one can also approximate this term by $\Mcal(\bw^{n,i}) \cdot \bth^{-1}\Mcal( \br^{n,i-1})$, and the
analysis presented next remains true and follows exactly the same lines.
 The complexity in this algorithm is that it requires solving 
 a large system generated by~\eqref{Lscheme_mono}, which   combines equations varied in type, and this is at each iteration $i>1$. Thus, encouraging the development of efficient techniques for  the resolution of these coupled systems.
\subsection{The  partially decoupled schemes}\label{subsec:algo_two_levels}
In the second set of iterative schemes, we only decouple  either the flow (\textnormal{\textbf{F}}) or mechanics (\textnormal{\textbf{M}}) or  heat (\textnormal{\textbf{H}})  from   the remaining two processes, which are being solved monolithically. Thus,  we transform the monolithic solver (\textnormal{\textbf{HFM}})  into  a \textit{two-level} iterative approach in which  two  simpler 
subproblems are  solved sequentially. This setting delivers the following three iterative approaches.
\subsubsection{(HF-M): coupled heat/flow }\label{sect:algo2}
Decoupling the mechanics calculation from the coupled flow and heat flow calculation, the first \textit{two-level} iterative  scheme reads as follows: at the iteration $i>1$, do:

\qquad \textbullet \ \textbf{Step 1}: Given $(T^{n,i-1}, p^{n,i-1}, \bw^{n,i-1}, \bu^{n,i-1})$, find $(T^{n,i}, \br^{n,i}, p^{n,i}, \bw^{n,i})$ such that
\bse\label{coupled_heatflow}
	\begin{alignat}{2}
		\nonumber &(\psi_{L_{T}}(T^{n,i},p^{n,i},\bu^{n,i-1}), S)
		+ \tau c_{f}(\Mcal(\bw^{n,i-1}) \cdot \bth^{-1}\Mcal( \br^{n,i}), S)  \\
		 &\qquad  + \tau(\nabla \cdot \br^{n,i},S)= \tau(z^{n},S) +(\psi(T^{n-1},p^{n-1},\bu^{n-1}), S)+L_T(T^{n,i-1},S), \qquad 
		&&\forall S \in \Tcal_h, \label{scheme1}\\
		&(\bth^{-1} \br^{n,i}, \by) - (T^{n,i}, \nabla \cdot \by) = 0, \quad
		&&\forall \by \in \Rcal_h, \\
		\nonumber &(\phi_{L_{p}}(T^{n,i},p^{n,i},\bu^{n,i-1}), q)
		+ \tau(\nabla \cdot \bw^{n,i},q) = \tau(g^{n},q)\\
		&\qquad \qquad\qquad \qquad\qquad\qquad + (\phi(T^{n-1},p^{n-1},\bu^{n-1}),q)+ L_p(p^{n,i-1},q),
		&&\forall q \in \Pcal_h, \\
		&(\bK^{-1} \bw^{n,i}, \bz) - (p^{n,i},\nabla \cdot \bz) = 0, \hfill \quad
		&&\forall \bz \in \Wcal_h.
	\end{alignat}
	
\qquad \textbullet \ \textbf{Step 2}: Given $(p^{n,i}, T^{n,i})$, find the displacement $\bu^{n,i}$ such that		
	\begin{alignat}{2}
		&2\mu(\bep(\bu^{n,i}),\bep(\bv)) + \lambda (\nabla \cdot \bu^{n,i}, \nabla \cdot \bv)= (\bff^{n},\bv) + (\beta T^{n,i} +\alpha p^{n,i}, \nabla \cdot \bv),
		\quad \quad  
		&&\qquad\,\forall \bv \in \Ucal_h.  \label{schemen}
	\end{alignat}
\ese

\subsubsection{(HM-F): coupled heat/mechanics}\label{sect:algo3}
The second   scheme  in this subsection is obtained by decoupling the flow calculation from the remaining coupled thermo-elasticity    calculation. This iterative  scheme reads: at the iteration $i>1$,  do:

\qquad \textbullet \ \textbf{Step 1}: Given $(T^{n,i-1}, p^{n,i-1}, \bw^{n,i-1}, \bu^{n,i-1})$, find $(T^{n,i}, \br^{n,i}, \bu^{n,i})$ such that
\bse
	\begin{alignat}{2}
		\nonumber &(\psi_{L_{T}}(T^{n,i},p^{n,i-1},\bu^{n,i}), S)
		+ \tau c_{f}(\Mcal(\bw^{n,i-1}) \cdot \bth^{-1}\Mcal( \br^{n,i}), S)  \\
		 &\qquad + \tau(\nabla \cdot \br^{n,i},S)= \tau(z^{n},S) + (\psi(T^{n-1},p^{n-1},\bu^{n-1}), S)
		+L_T(T^{n,i-1},S), \qquad 
		&&\forall S \in \Tcal_h,\\
		&(\bth^{-1} \br^{n,i}, \by) - (T^{n,i}, \nabla \cdot \by) = 0, \quad
		&&\forall \by \in \Rcal_h, \\
		\nonumber&2\mu(\bep(\bu^{n,i}),\bep(\bv)) + \lambda (\nabla \cdot \bu^{n,i}, \nabla \cdot \bv)\\
		&\qquad \qquad\qquad - \beta (T^{n,i}, \nabla \cdot \bv)= (\bff^{n},\bv) + \alpha (p^{n,i-1}, \nabla \cdot \bv), 
		\qquad \qquad \qquad \qquad 
		&&\forall \bv \in \Ucal_h.
	\end{alignat}
		
\qquad \textbullet \ \textbf{Step 2}: Given $(T^{n,i}, \bu^{n,i}, p^{n,i-1})$,  find $(p^{n,i}, \bw^{n,i})$ such that		
	\begin{alignat}{2}
		\nonumber &(c_0 + L_p)(p^{n,i}, q) 
		+ \tau(\nabla \cdot \bw^{n,i},q)= \tau(g^{n},q) + (\phi(T^{n-1},p^{n-1},\bu^{n-1}), q) \qquad\\
		&\qquad \qquad+ L_p(p^{n,i-1},q) + b_0(T^{n,i}- \alpha\nabla \cdot \bu^{n,i}, q), \qquad \qquad && \qquad\forall q \in \Pcal_h, \\
		&(\bK^{-1} \bw^{n,i}, \bz) - (p^{n,i},\nabla \cdot \bz) = 0, \qquad \qquad 
		&&\qquad \forall \bz \in \Wcal_h.
	\end{alignat}
\ese

\subsubsection{(FM-H): coupled flow/mechanics}\label{sect:algo4}
The last  \textit{two-level} scheme is obtained by decoupling the poro-elasticity (solved monolithically) calculation from the heat flow. This iterative  scheme reads: at the iteration $i>1$, do:

\qquad \textbullet \ \textbf{Step 1}: Given $(p^{n,i-1}, \bu^{n,i-1}, T^{n,i-1})$, find $(p^{n,i}, \bw^{n,i}, \bu^{n,i})$ such that
\bse
	\begin{alignat}{2}
		\nonumber &(\phi_{L_{p}}(T^{n,i-1},p^{n,i},\bu^{n,i}), q)
		 + \tau(\nabla \cdot \bw^{n,i},q) = \tau(g^{n},q)\\
		&\qquad \qquad\qquad \qquad + (\phi(T^{n-1},p^{n-1},\bu^{n-1}), q)+L_p(p^{n,i-1},q), 
		&&\qquad \qquad \qquad \forall q \in \Pcal_h,\\
		&(\bK^{-1} \bw^{n,i}, \bz) - (p^{n,i}, \nabla \cdot \bz) = 0, \quad
		&&\qquad \qquad \qquad \forall \bz \in \Wcal_h, \\
		\nonumber&2\mu(\bep(\bu^{n,i}),\bep(\bv)) + \lambda (\nabla \cdot \bu^{n,i}, \nabla \cdot \bv) 
		\\
		&\qquad \qquad\qquad \qquad - \alpha (p^{n,i}, \nabla \cdot \bv)= (\bff^{n},\bv) + \beta (T^{n,i-1}, \nabla \cdot \bv), 		
		&&\qquad \qquad \qquad \forall \bv \in \Ucal_h. 
	\end{alignat}
		
\qquad \textbullet \ \textbf{Step 2}: Given $(p^{n,i}, \bw^{n,i}, \bu^{n,i}, T^{n,i-1})$, find $(T^{n,i}, \br^{n,i})$ such that		
	\begin{alignat}{2}
		\nonumber &(a_0 + L_T)(T^{n,i}, S) +\tau c_{f}(\Mcal(\bw^{n,i}) \cdot \bth^{-1}\Mcal( \br^{n,i}), S)
		+ \tau(\nabla \cdot \br^{n,i},S)  \\
		\nonumber &\qquad \qquad= \tau(z^{n},S) + (\psi(T^{n-1},p^{n-1},\bu^{n-1}), S) \\
		&\qquad \qquad\qquad  + L_T(T^{n,i-1},S) + b_0(p^{n,i}, S) - \beta(\nabla \cdot \bu^{n,i}, S), 
		&& \quad\qquad \quad \forall S \in \Tcal_h, \\
		&(\bth^{-1} \br^{n,i}, \by) - (T^{n,i},\nabla \cdot \by) = 0, \hfill 
		&&\quad \qquad \quad \forall \by \in \Rcal_h.
	\end{alignat}
\ese

\subsection{The  fully decoupled schemes}\label{subsec:algo_three_levels}
In this set of iterative coupling schemes, we simply split the three processes, providing three sub-problems to be solved sequentially. Fixing  the mechanics calculation in the third level, two approaches are then derived in  which either the problem of flow or the heat is solved  first followed by solving the other system and then the  mechanics using the already calculated information,   leading  to recover the original solution. These schemes enjoy the solving of much simpler subsystems  through the algorithm, as well as the facility to reuse existing codes for each component of the problem.
\subsubsection{(H-F-M): decoupled heat - flow - mechanics}\label{sect:algo5}

At each iteration all three subproblems are decoupled, and are solved in the order heat $\rightarrow$ flow $\rightarrow$ mechanics. This iterative  scheme reads: at the iteration $i>1$, do:

\bse\label{decoupled1}
\qquad \textbullet \ \textbf{Step 1}: Given $(p^{n,i-1}, \bw^{n,i-1}, T^{n,i-1}, \bu^{n,i-1})$ find $(T^{n,i}, \br^{n,i})$ such that
	\begin{alignat}{2}
		\nonumber &(\psi_{L_{T}}(T^{n,i},p^{n,i-1},\bu^{n,i-1}), S)  
		+ \tau c_{f}(\Mcal(\bw^{n,i}) \cdot \bth^{-1}\Mcal( \br^{n,i}), S)  \\
		 &\qquad  + \tau(\nabla \cdot \br^{n,i},S)= \tau(z^{n},S) + (\psi(T^{n-1},p^{n-1},\bu^{n-1}), S)
		+L_T(T^{n,i-1},S), \qquad  
		&& \forall S \in \Tcal_h, \label{decoupled1_1}\\
		&(\bth^{-1} \br^{n,i}, \by) - (T^{n,i}, \nabla \cdot \by) = 0, 
		&& \forall \by \in \Rcal_h.
	\end{alignat}

\qquad \textbullet \ \textbf{Step 2}: Given $(p^{n,i-1}, T^{n,i}, \bu^{n,i-1})$ find $(p^{n,i}, \bw^{n,i})$ such that
	\begin{alignat}{2}
		\nonumber &(c_0 + L_p)(p^{n,i}, q)   
		+ \tau(\nabla \cdot \bw^{n,i},q) \\
		\nonumber &\qquad \qquad= \tau(g,q) + (\phi(T^{n-1},p^{n-1},\bu^{n-1}), q) \\
		&\qquad \qquad\qquad\quad+ L_p(p^{n,i-1},q) + b_0(T^{n,i}, q) - \alpha(\nabla \cdot \bu^{n,i-1}, q),
		&&\qquad \qquad \qquad\quad \forall q \in \Pcal_h, \\
		&(\bK^{-1} \bw^{n,i}, \bz) - (p^{n,i},\nabla \cdot \bz) = 0, 
		&&  \qquad\qquad \qquad\quad\forall \bz \in \Wcal_h.\label{decoupled1_5}
	\end{alignat}
	
\qquad \textbullet \ \textbf{Step 3}: Given $(p^{n,i}, T^{n,i})$ find $\bu^{n,i}$ such that		
	\begin{alignat}{2}
		&2\mu(\bep(\bu^{n,i}),\bep(\bv)) + \lambda (\nabla \cdot \bu^{n,i}, \nabla \cdot \bv) = (\bff,\bv) + (\beta T^{n,i}+ \alpha p^{n,i}, \nabla \cdot \bv), \qquad \qquad 
		&&   \forall \bv \in \Ucal_h. 
	\end{alignat}
\ese

\subsubsection{(F-H-M): decoupled flow - heat - mechanics}\label{sect:algo6}

At each iteration all three subproblems are decoupled, and are solved in the order flow $\rightarrow$ heat $\rightarrow$ mechanics. This iterative  scheme reads: at iteration $i>1$, do:

\qquad \textbullet \ \textbf{Step 1}: Given $(p^{n,i-1}, T^{n,i-1}, \bu^{n,i-1})$ find $(p^{n,i}, \bw^{n,i})$ such that
\bse\label{decoupled2}
	\begin{alignat}{2}
		\nonumber &(\phi_{L_{p}}(T^{n,i-1},p^{n,i},\bu^{n,i-1}), q)   
		+ \tau(\nabla \cdot \bw^{n,i},q) \\
		 &\qquad \qquad= \tau(g,q) +(\phi(T^{n-1},p^{n-1},\bu^{n-1}), q)  + L_p(p^{n,i-1},q), \qquad \qquad\qquad\quad
		&&\forall q \in \Pcal_h, \label{decoupled2_1}\\
		&(\bK^{-1} \bw^{n,i}, \bz) - (p^{n,i},\nabla \cdot \bz) = 0, \qquad \qquad 
		&& \forall \bz \in \Wcal_h.
	\end{alignat}
	
\qquad \textbullet \ \textbf{Step 2}: Given $(p^{n,i}, \bw^{n,i}, T^{n,i-1}, \bu^{n,i-1})$, find $(T^{n,i}, \br^{n,i})$ such that
	\begin{alignat}{2}
		\nonumber &(a_0 + L_T)(T^{n,i}, S)  
		+ \tau c_{f}(\Mcal(\bw^{n,i}) \cdot \bth^{-1}\Mcal( \br^{n,i}), S) + \tau(\nabla \cdot \br^{n,i},S)\qquad\qquad \\
		\nonumber &\qquad \qquad= \tau(h,S) + (\psi(T^{n-1},p^{n-1},\bu^{n-1}), S)\\
		&\qquad\qquad \qquad+L_T(T^{n,i-1},S) + b_0(p^{n,i},S) - \beta(\nabla \bu^{n,i-1}, S), \qquad \qquad \qquad
		&&\forall S \in \Tcal_h,\\
		&(\bth^{-1} \br^{n,i}, \by) - (T^{n,i}, \nabla \cdot \by) = 0,
		&& \forall \by \in \Rcal_h.\label{decoupled2_5}
	\end{alignat}
	
\qquad \textbullet \ \textbf{Step 3}: Given $(p^{n,i}, T^{n,i})$, find $\bu^{n,i}$ such that		
	\begin{alignat}{2}
		&2\mu(\bep(\bu^{n,i}),\bep(\bv)) + \lambda (\nabla \cdot \bu^{n,i}, \nabla \cdot \bv) = (\bff^{n},\bv) + (\beta T^{n,i} +\alpha p^{n,i}, \nabla \cdot \bv),\qquad
		&&\quad\forall \bv \in \Ucal_h.  
	\end{alignat}
\ese

\section{Convergence analysis}\label{sec:conv}
The starting point for our analysis is the existence and uniqueness of a solution to~\eqref{discretevar}. To this aim, we will make use of the following Lemma (cf.~\cite{MR2116915}), stating the   Lipschitz property of the cut-off operator $\Mcal$:
\begin{lemma}[Property of $\Mcal$]
The \textquotedblleft cut-off \textquotedblright operator $\Mcal$ defined
as in equation~\eqref{cutoffoperator} is uniformly Lipschitz continuous,
\begin{equation}
\norm{\Mcal(\bz_1) - \Mcal(\bz_2)}_{(L^\infty(\Omega))^d}
\leq \norm{\bz_1 - \bz_2}_{(L^\infty(\Omega))^d}.
\end{equation}
\end{lemma}
Thus, we have 
\bse\begin{alignat}{2}
\norm{\Mcal({\bw}^{n}) - \Mcal(\bw^{n,i})}_{(L^\infty(\Omega))^d}
\leq \norm{{\bw}^n - \bw^{n,i}}_{(L^\infty(\Omega))^d}, 
\end{alignat}
and 
\begin{alignat}{2}
\norm{\Mcal(\bw^n)}_{(L^\infty(\Omega))^d}
\leq M. 
\end{alignat}\ese
The proof of the next  Theorem is based on showing that  the scheme~\eqref{Lscheme_mono} is a contraction, and then by applying the Banach fixed-point theorem~\cite{cheney2013analysis}, to deduce convergence of the scheme. In what follows we will frequently use the following polarization and binomial identities,
\begin{equation}\label{identities}
4(u,v) = \norm{u+v}^2 - \norm{u-v}^2, \quad \textnormal{ and } \quad
2(u-v,u) = \norm{u}^2 + \norm{u-v}^2 - \norm{v}^2.
\end{equation}
Finally, we define the difference functions between the solutions at the iteration $i$ and $i-1$ of problem~\eqref{Lscheme_mono}, respectively as 
\begin{equation}\label{diffdef}
(e_T^{i}, \be_\br^{i}, e_p^{i}, \be_\bw^{i}, \be_\bu^{i}) := (T^{n,i}-T^{n,i-1}, \br^{n,i} - \br^{n,i-1}, p^{n,i}-p^{n,i-1}, \bw^{n,i} - \bw^{n,i-1}, \bu^{n,i}-\bu^{n,i-1}). 
\end{equation}
With this, we state the first of our main results:
\begin{theorem}[Convergence of the monolithic $L$-scheme~\textnormal{\textbf{HFM}}]\label{thm:wellposedness}Assuming that~\textnormal{\Assum}  hold true, and the time step is  small enough, i.e. $\tau < \dfrac{2(a_{0}-b_{0})}{c_f^2 M^2\left(\dfrac{k_M}{\theta_m} + 1\right) -  \dfrac{\theta_m}{4c_{\Omega,d}}}$, then, the monolithic $L$-scheme~\textnormal{\textbf{HFM}} (Algorithm~\ref{sect:algo1}) defines a contraction satisfying
\begin{align}
\nonumber &\left(a_0 -b_0 + \frac{L_T}{2} + \frac{\tau\theta_m}{4c_{\Omega,d}} - \frac{\tau c_f^2 M^2}{2}\left(\frac{k_M}{\theta_m} + 1 \right)\right) \norm{e_T^{i}}^2  +\left(c_0 -b_0 + \frac{L_p}{2} \right) \norm{e_p^{i}}^2\\
\nonumber&\qquad\quad\qquad  + \tau \norm{\be_{\bw}^{i}}^2_{\bK^{-1}} +\frac{\tau}{2} \norm{\be_{\br}^{i}}^2_{\bth^{-1}}+2\mu \norm{\bep(\be_{\bu}^{i})}^2 + \lambda \norm{\nabla \cdot \be_{\bu}^{i}}^2\\
&\quad\qquad\quad\quad\qquad\leq \frac{L_T}{2} \norm{e_T^{i-1}}^2 + \frac{L_p}{2} \norm{e_p^{i-1}}^2 
+ \frac{\tau}{2} \norm{\be_{\bw}^{i-1}}^2_{\bK^{-1}}. 
\end{align}
Therefrom, the limit is the unique solution of the problem~\eqref{discretevar}.
\end{theorem}
\begin{remark}[Bound on time step]
Note that $a_0 - b_0 > 0$ due to the Assumption~\ref{constraints}, and $c_f^2M^2\left(\dfrac{k_M}{\theta_m} + 1\right) - \dfrac{\theta_m}{4c_{\Omega,d}} > 0$ by the choice of $M$ large enough.  
\end{remark}
\begin{proof}

We begin by deriving the error equations satisfied by $(e_T^{i},\be_\br^{i},e_p^{i},\be_\bw^{i},\be_\bu^{i})$, i.e. subtract the equations~\eqref{Lscheme_mono} for $i$ from the ones for $i-1$, and obtain
\bse
	\begin{alignat}{2}
		 \nonumber &(\psi_{L_{T}}(e_T^{i},e_p^{i},\be_\bu^{i}), S) + \tau(\nabla \cdot \be_\br^{n},S) + \tau c_f(\Mcal(\bw^{n,i-1})  \cdot \bth^{-1} [\Mcal(\br^{n,i}) - \Mcal(\br^{n,i-1})], S)\\
		 &\qquad\qquad+ \tau c_f([\Mcal(\bw^{n,i-1}) - \Mcal(\bw^{n,i-2})] \cdot \bth^{-1} \Mcal(\br^{n,i}), S)  = L_T(e_T^{i-1},S), \qquad
		&&\forall S \in \Tcal_h, \label{emon1}\\
		&(\bth^{-1} \be_\br^{i}, \by) - (e_T^{i}, \nabla \cdot \by) = 0, \quad
		&&\forall \by \in \Rcal_h, \label{emon2}\\
		 &(\phi_{L_{p}}(e_T^{i},e_p^{i},\be_\bu^{i}), q)
		+ \tau(\nabla \cdot \be_\bw^{i},q)  = L_{p}(e_p^{i-1},q), \hfill
		&&\forall q \in \Pcal_h, \\
		&(\bK^{-1} \be_\bw^{i}, \bz) - (e_p^{i},\nabla \cdot \bz) = 0, \hfill 
		&&\forall \bz \in \Wcal_h,\\
		&2\mu(\bep(\be_\bu^{i}),\bep(\bv)) + \lambda (\nabla \cdot  \be_\bu^{i}, \nabla \cdot \bv) -  (\beta e_T^{i} +\alpha e_p^{i}, \nabla \cdot \bv)= 0,
		&&\forall \bv \in \Ucal_h.\label{emon5}
	\end{alignat}
\ese
We choose now $S = e_T^{i},\, \by = \tau \be_{\br}^{i},\, q = e_p^{i},\, \bz = \tau \be_{\bw}^{i},$ and $\bv = \be_{\bu}^{i}$ as test functions in equations \eqref{emon1}-- \eqref{emon5}, respectively. Then, summing the  resulting equations and using the identity~\eqref{identities} together with applying Cauchy-Schwarz and Young inequalities    and  some algebraic manipulations, we get, for any $\epsilon_1, \epsilon_2 > 0$, 
\begin{align}
\nonumber &\left(a_0 -b_0 + \frac{L_T}{2} \right) \norm{e_T^{i}}^2 + \tau \norm{\be_{\br}^{i}}^2_{\bth^{-1}}
+\left(c_0 -b_0 + \frac{L_p}{2} \right) \norm{e_p^{i}}^2 + \tau \norm{\be_{\bw}^{i}}^2_{\bK^{-1}} +2\mu \norm{\bep(\be_{\bu}^{i})}^2 + \lambda \norm{\nabla \cdot \be_{\bu}^{i}}^2\\
\nonumber&\quad\leq \frac{L_T}{2} \norm{e_T^{i-1}}^2 + \frac{L_p}{2} \norm{e_p^{i-1}}^2 + \tau c_f \norm{\Mcal(\bw^{i-1}) \cdot \bth^{-1} \be_{\br}^{i}} \norm{e_T^{i}}+\tau c_f \norm{\be_{\bw}^{i-1} \cdot \bth^{-1} \Mcal(\br^{i-1})} \norm{e_T^{i}}\\
&\quad\leq \frac{L_T}{2} \norm{e_T^{i-1}}^2 + \frac{L_p}{2} \norm{e_p^{i-1}}^2
+ \tau c_f M\left(\frac{\epsilon_1}{2} + \frac{\epsilon_2}{2} \right) \norm{e_T^{i}}^2+ \tau c_f M \frac{1}{2\epsilon_1}\norm{\be_{\br}^{i}}^2_{\bth^{-1}}
+ \tau c_f M \frac{k_M}{\theta_m} \frac{1}{2\epsilon_2} \norm{\be_{\bw}^{i-1}}^2_{\bK^{-1}}. \label{monest1}
\end{align}
From equation \eqref{emon2}, and by Thomas' lemma \cite{thomas1977analyse}, there exists $\hat{\by} \in \Rcal_h$ and a constant $c_{\Omega,d} >0$ depending only on the domain and spatial dimension such that $\nabla \cdot \hat{\by} = e_T^i$ with $\norm{\hat{\by}} \leq c_{\Omega,d} \norm{e_T^i}$. Thus, taking $\hat{\by}$ as a test function in \eqref{itere4} we deduce
\begin{align}
\nonumber \norm{e_T^i}^2 = (e_T^i, \nabla \cdot \hat{\by}) &= (\bth^{-1} \be_{\br}^i, \hat{\by}) \\
\nonumber &\leq \norm{\be_{\br}^i}_{\bth^{-1}} \cdot \frac{1}{\sqrt{\theta_m}} \norm{\hat{\by}} \\
&\leq \norm{\be_{\br}^i}_{\bth^{-1}} \cdot \frac{c_{\Omega,d}}{\sqrt{\theta_m}} \norm{e_T^i}, 
\end{align}
which leads to
\begin{equation}
\frac{\theta_m}{c_{\Omega,d}} \norm{e_T^i}^2 \leq \norm{\be_{\br}^i}^2_{\bth^{-1}}.\label{thomasest}
\end{equation}
Replacing \eqref{thomasest} in \eqref{monest1} while choosing $\epsilon_1 = c_fM$ and $\epsilon_2 = c_f M k_M/\theta_m$, we obtain
\begin{align}
\nonumber &\left(a_0 -b_0 + \frac{L_T}{2} + \frac{\tau \theta_m}{4c_{\Omega,d}} - \frac{\tau c_f^2 M^2}{2}\left(\frac{k_M}{\theta_m} + 1\right)\right) \norm{e_T^{i}}^2+2\mu \norm{\bep(\be_{\bu}^{i})}^2 + \lambda \norm{\nabla \cdot \be_{\bu}^{i}}^2  \\
\nonumber&\qquad +\left(c_0 -b_0 + \frac{L_p}{2} \right) \norm{e_p^{i}}^2 + \tau \norm{\be_{\bw}^{i}}^2_{\bK^{-1}} +\frac{\tau}{4} \norm{\be_{\br}^{i}}^2_{\bth^{-1}}\\
&\quad\qquad\quad\leq \frac{L_T}{2} \norm{e_T^{i-1}}^2 + \frac{L_p}{2} \norm{e_p^{i-1}}^2 
+ \frac{\tau}{2} \norm{\be_{\bw}^{i-1}}^2_{\bK^{-1}}. \label{monolithicestimate}
\end{align}
The contraction of the residuals follows if the time step $\tau$ is small enough, i.e.,
\begin{equation}
\tau < \dfrac{2(a_{0}-b_{0})}{c_f^2 M^2\left(\dfrac{k_M}{\theta_m} + 1\right) -  \dfrac{\theta_m}{4c_{\Omega,d}}},
\end{equation}
where $c_{\Omega,d}$ is not bounded from above, and can thus always be chosen such that the denominator in the above is positive. This proves the convergence of the monolithic $L$-scheme. The limit is then the unique solution of~\eqref{discretevar}.
\end{proof}
The well-posedness of the discrete variational problem~\eqref{discretevar} is established by the  Theorem~\ref{thm:wellposedness}, where the 
solution at time $t^n$, $n\leq0$, is denoted  by 
$(T^n, \br^n, p^n, \bw^n, \bu^n)$. Thus,  we can now  prove the convergence of the decoupled schemes to this solution. We begin with analyzing the partially decoupled schemes, introduced in Subsection~\ref{subsec:algo_two_levels}. To this end, we let the difference functions defined in~\eqref{diffdef} now be the differences between the solutions at the iteration $i$ of problem~\eqref{coupled_heatflow}, and the solutions to~\eqref{discretevar}, i.e.
\begin{equation}\label{diffdef2}
(e_T^{i}, \be_\br^{i}, e_p^{i}, \be_\bw^{i}, \be_\bu^{i}) := (T^{n,i}-T^{n}, \br^{n,i} - \br^{n}, p^{n,i}-p^{n}, \bw^{n,i} - \bw^{n}, \bu^{n,i}-\bu^{n}). 
\end{equation}
The second of our main results is given through
\begin{theorem}[Convergence of the partially  decoupled schemes]\label{mainthm} Assuming that~\textnormal{\Assum}  hold true,  the stabilization parameters are such that 
\begin{equation}\label{Lbound}
L_p \geq \dfrac{4\alpha^2}{3(\frac{2\mu}{d} + \lambda)}, \quad \textnormal{ and } \quad
L_T \geq \dfrac{4\beta^2}{3(\frac{2\mu}{d} + \lambda)},
\end{equation}
and the  time step is small enough such that  $\tau < \dfrac{2(a_0 - b_0)}{c_f^2 M^2(\dfrac{k_M}{\theta_m} + 1) - \dfrac{\theta_m}{4c_{\Omega,d}}}$. Then, the partially decoupled L-scheme~\textnormal{\textbf{HF-M}} (Algorithm~\ref{sect:algo2}) is a contraction  given by
\begin{alignat}{2}
\nonumber& \left(a_0 -  b_0 + \frac{L_T}{2} + \frac{\tau \theta_m}{4c_{\Omega,d}}- \frac{\tau c_f^2M^2}{2}\left(\frac{k_M}{\theta_m} + 1\right)\right) \norm{e_T^i}^2 
\\
\nonumber&\qquad\qquad+\left(c_0 - b_0 + \frac{L_p}{2} \right) \norm{e_p^i}^2  
+ \tau \norm{\be_{\bw}^i}^2_{\bK^{-1}} + \dfrac{\tau}{4} \norm{\be_{\br}^i}^2_{\bth^{-1}}\\
\label{est4}&\qquad\qquad\qquad \leq \frac{L_T}{2}\norm{e_T^{i-1}}^2
+ \frac{L_p}{2}\norm{e_p^{i-1}}^2 +\frac{\tau}{2} \norm{\be_{\bw}^{i-1}}^2_{\bK^{-1}}.
\end{alignat}
Furthemore, there holds, 
\begin{alignat}{2}
\label{est42}&\frac{\mu}{2}\norm{\bep(\be_{\bu}^i)}^2 + \frac{\lambda}{4} \norm{\nabla \cdot \be_{\bu}^i}^2
\leq \frac{2\alpha^2}{3(\frac{2\mu}{d} + \lambda)} \norm{e_p^i}^2
+ \frac{2\beta^2}{3(\frac{2\mu}{d} + \lambda)} \norm{e_T^i}^2.
\end{alignat}
\end{theorem}
\begin{proof}
We start by taking the difference of equations \eqref{scheme1} -- \eqref{schemen} at iteration $i$ with the corresponding equations solved by $({T}^n, {\br}^n, {p}^n, {\bw}^n, {\bu}^n)$. This leads to the following set of difference equations
\bse
\begin{alignat}{2}
	&(c_0 + L_p)(e_p^{i}, q) - b_0(e_T^{i}, q) + \tau(\nabla \cdot \be_{\bw}^{i}, q) 
	+ \alpha(\nabla \cdot \be_{\bu}^{i-1}, q)= L_p(e_p^{i-1},q),
	&&\quad \forall q \in \Pcal_h,\label{itere1}\\
	&(\bK^{-1} \be_{\bw}^{i}, \bz) - (e_p^{i}, \nabla \cdot \bz) = 0,&&\quad \forall \bz \in \Wcal_h,\label{itere2}\\
	\nonumber &(a_0+L_T)(e_T^{i}, S) - b_0(e_p^i, S) + \tau(\nabla \cdot \be_{\br}^{i}, S)
	+ \tau c_f([\Mcal(\bw^{n,i-1}) - \Mcal({\bw}^n)] \cdot \bth^{-1} {\br}^n, S) \\
	&\quad+ \tau c_f(\Mcal(\bw^{n,i-1}) \cdot \bth^{-1} [\Mcal(\br^{n,i}) - \Mcal({\br}^{n})], S)   
	+ \beta(\nabla \cdot \be_{\bu}^{i-1},S) = L_T(e_T^{i-1}, S) ,
	&&\quad\forall S \in \Tcal_h, \label{itere3} \\
	&(\bth^{-1} \be_{\br}^{i}, \by) - (e_T^{i}, \nabla \cdot \by) = 0,&&\quad\forall \by \in \Rcal_h \label{itere4}\\
	&2\mu(\bep(\be_{\bu}^i), \bep(\bv)) + \lambda(\nabla \cdot \be_{\bu}^i, \nabla \cdot \bv) 
	- (\alpha e_p^i + \beta e_T^i, \nabla \cdot \bv) = 0,&&\quad \forall \bv \in \Ucal_h, \label{itere5}
\end{alignat}
\ese
where we used the following identity
\begin{align}
\nonumber&(\Mcal(\bw^{n,i-1}) \cdot \bth^{-1} \Mcal(\br^{n,i}), S) - (\Mcal({\bw}^{n}) \cdot \bth^{-1} \Mcal({\br}^{n}), S)\\
&\quad = 
([\Mcal(\bw^{n,i-1}) - \Mcal({\bw}^n)] \cdot \bth^{-1} \Mcal({\br}^n), S) 
+ (\Mcal(\bw^{n,i-1}) \cdot \bth^{-1} [\Mcal(\br^{n,i}) - \Mcal({\br}^n)], S).
\end{align}
The aim now is to show a contraction of successive error functions, thereby implying convergence of the sequences $(T^{n,i}, \br^{n,i}, p^{n,i}, \bw^{n,i}, \bu^{n,i})$ as $i \rightarrow \infty$ for $n \geq 1$, by the Banach Fixed Point Theorem~\cite{cheney2013analysis}. Taking as test functions $q = e_p^i, \bz = \tau \be_{\bw}^i, S = e_T^i, \by = \tau \be_{\br}^i$, and $\bv = \be_{\bu}^{i-1}$ in \eqref{itere1} -- \eqref{itere5}, respectively, and adding the resulting equations together, we obtain
\begin{align}
\nonumber&c_0 \norm{e_p^i}^2 + \tau \norm{\be_{\bw}^i}^2_{\bK^{-1}} 
+ a_0 \norm{e_T^i}^2 + \tau \norm{\be_{\br}^i}^2_{\bth^{-1}} \\
\nonumber &\quad+\frac{1}{2} L_p \left(\norm{e_p^i}^2 + \norm{e_p^i - e_p^{i-1}}^2 - \norm{e_p^{i-1}}^2 \right)
+ \frac{1}{2} L_T \left(\norm{e_T^i}^2 + \norm{e_T^i - e_T^{i-1}}^2 - \norm{e_T^{i-1}}^2 \right) \\
\nonumber &\qquad+2\mu \frac{1}{4} \norm{ \bep(\be_{\bu}^i + \be_{\bu}^{i-1})}^2 
+ \lambda \frac{1}{4} \norm{ \nabla \cdot (\be_{\bu}^i + \be_{\bu}^{i-1})}^2\\
\nonumber&\qquad \quad= 2b_0 (e_T^i, e_p^i) + 2\mu\frac{1}{4} \norm{\bep(\be_{\bu}^i - \be_{\bu}^{i-1})}^2 
+ \lambda \frac{1}{4}\norm{\nabla \cdot(\be_{\bu}^i - \be_{\bu}^{i-1})}^2 \\
&\qquad \qquad -\tau c_f([\Mcal(\bw^{n,i-1}) - \Mcal({\bw}^n)] \cdot \bth^{-1} \Mcal({\br}^n), e_T^i) 
- \tau c_f (\Mcal(\bw^{n,i-1}) \cdot \bth^{-1} [\Mcal(\br^{n,i}) - \Mcal({\br}^n)], e_T^i), \label{est1}
\end{align}
where we used the identities~\eqref{identities}. On the other hand, by taking the difference of eq. \eqref{itere5} at iteration $i$ and $i-1$, testing with $\be_{\bu}^{i} - \be_{\bu}^{i-1}$, and using the Cauchy-Schwarz inequality we get
\begin{align}
\nonumber &2\mu \norm{\bep(\be_{\bu}^i - \be_{\bu}^{i-1})}^2 + \lambda \norm{\nabla \cdot (\be_{\bu}^i - \be_{\bu}^{i-1})}^2 \\
&\quad=\alpha(e_p^i - e_p^{i-1}, \nabla \cdot (\be_{\bu}^i - \be_{\bu}^{i-1})) 
+ \beta(e_T^i - e_T^{i-1}, \nabla \cdot (\be_{\bu}^i - \be_{\bu}^{i-1}))\\
&\qquad \leq \left(\alpha \norm{e_p^i - e_p^{i-1}} + \beta \norm{e_T^i - e_T^{i-1}}\right) \norm{\nabla \cdot (\be_{\bu}^i - \be_{\bu}^{i-1})}.
\end{align}
Let now $\xi \in (0,1)$ and rewrite the above estimate as
\begin{align}
\nonumber &2\mu \norm{\bep(\be_{\bu}^i - \be_{\bu}^{i-1})}^2 + \lambda \norm{\nabla \cdot (\be_{\bu}^i - \be_{\bu}^{i-1})}^2 \\
&\quad \leq \left(\alpha \norm{e_p^i - e_p^{i-1}} + \beta \norm{e_T^i - e_T^{i-1}}\right) 
\left( \xi \sqrt{d} \norm{\bep(\be_{\bu}^i - \be_{\bu}^{i-1})} + (1-\xi) \norm{\nabla \cdot (\be_{\bu}^i - \be_{\bu}^{i-1})} \right).
\end{align}
We now follow \cite{both2017robust} and choose $\xi = \dfrac{2\mu}{2\mu + d\lambda}$, which together with the Young inequality yields
\begin{equation}\label{ubound}
\frac{\mu}{2} \norm{\bep(\be_{\bu}^i - \be_{\bu}^{i-1})}^2 + \frac{\lambda}{4} \norm{\nabla \cdot (\be_{\bu}^i - \be_{\bu}^{i-1})}^2
\leq \frac{2 \alpha^2}{3(\frac{2\mu}{d} + \lambda)} \norm{e_p^i - e_p^{i-1}}^2 
+ \frac{2\beta^2}{3(\frac{2\mu}{d} + \lambda)} \norm{e_T^i - e_T^{i-1}}^2.
\end{equation}
Combining this with eq. \eqref{est1} leads to
\begin{align}
\nonumber&(c_0 + \frac{L_p}{2}) \norm{e_p^i}^2 + \tau \norm{\be_{\bw}^i}^2_{\bK^{-1}} 
+ (a_0 + \frac{L_T}{2}) \norm{e_T^i}^2 + \tau \norm{\be_{\br}^i}^2_{\bth^{-1}} 
+\frac{\mu}{2} \norm{ \bep(\be_{\bu}^i + \be_{\bu}^{i-1})}^2 
+ \frac{\lambda}{4} \norm{ \nabla \cdot (\be_{\bu}^i + \be_{\bu}^{i-1})}^2 \\
\nonumber &\quad+ \left(\frac{L_p}{2} - \frac{2 \alpha^2}{3(\frac{2\mu}{d} + \lambda)} \right)\norm{e_p^i - e_p^{i-1}}^2
+ \left(\frac{L_T}{2} - \frac{2\beta^2}{3(\frac{2\mu}{d} + \lambda)} \right)\norm{e_T^i - e_T^{i-1}}^2\\
\nonumber&\qquad \leq 2b_0 (e_T^i, e_p^i) + \frac{L_p}{2}\norm{e_p^{i-1}}^2 +\frac{L_T}{2}\norm{e_T^{i-1}}^2
-\tau c_f([\Mcal(\bw^{n,i-1}) - \Mcal({\bw}^n)] \cdot \bth^{-1} \Mcal({\br}^n), e_T^i) \\
&\qquad \qquad- \tau c_f (\Mcal(\bw^{n,i-1}) \cdot \bth^{-1} [\Mcal(\br^{n,i}) - \Mcal({\br}^n)], e_T^i). \label{est2}
\end{align}
We thus need to impose some constraints on the stabilization parameters, i.e. $L_p \geq \dfrac{4\alpha^2}{3(\frac{2\mu}{d} + \lambda)}$ and $L_T \geq \dfrac{4\beta^2}{3(\frac{2\mu}{d} + \lambda)}$. With this, we can discard some positive terms on the left hand side of \eqref{est2}, and use the Cauchy-Schwarz and Young inequalities, together with the Lipschitz property of $\Mcal$ to obtain
\begin{align}
\nonumber&\left(c_0 - b_0 + \frac{L_p}{2} \right) \norm{e_p^i}^2 + \tau \norm{\be_{\bw}^i}^2_{\bK^{-1}} 
+ \left(a_0 - b_0 + \frac{L_T}{2} - \tau c_fM(\frac{\epsilon_1}{2} + \frac{\epsilon_2}{2})\right) \norm{e_T^i}^2 
+ \tau \norm{\be_{\br}^i}^2_{\bth^{-1}} \\
&\quad \leq \frac{L_p}{2}\norm{e_p^{i-1}}^2 +\frac{L_T}{2}\norm{e_T^{i-1}}^2
+ \tau c_f M\frac{k_M}{\theta_m}\frac{1}{2\epsilon_1}\norm{\be_{\bw}^{i-1}}^2_{\bK^{-1}}
+ \tau c_f M \frac{1}{2\epsilon_2} \norm{\be_{\br}^i}^2_{\bth^{-1}}, \label{est3}
\end{align}
for some $\epsilon_1, \epsilon_2 > 0$, and where $k_M$ and $\theta_m$ are given by \ref{Kdef} -- \ref{thdef}. From \eqref{itere4}, we obtain in the same way as in \eqref{thomasest}
\begin{equation}
\frac{\theta_m}{c_{\Omega,d}} \norm{e_T^i}^2 \leq \norm{\be_{\br}^i}^2_{\bth^{-1}}.\label{thomasest2}
\end{equation}
Replacing \eqref{thomasest2} in \eqref{est3} while choosing $\epsilon_1 = c_f M k_M/\theta_m$ and $\epsilon_2 = c_f M$, we get
\begin{align}
\nonumber&\left(a_0 - b_0 + \frac{L_T}{2} + \frac{\tau \theta_m}{4c_{\Omega,d}}- \frac{\tau c_f^2M^2}{2}\left(\frac{k_M}{\theta_m} + 1\right)\right) \norm{e_T^i}^2 + \frac{\tau}{4} \norm{\be_{\br}^i}^2_{\bth^{-1}} \\
\nonumber&\quad +\left(c_0 - b_0 + \frac{L_p}{2} \right) \norm{e_p^i}^2 + \tau \norm{\be_{\bw}^i}^2_{\bK^{-1}} \\
&\qquad \leq \frac{L_p}{2}\norm{e_p^{i-1}}^2 +\frac{L_T}{2}\norm{e_T^{i-1}}^2
+ \frac{\tau}{2} \norm{\be_{\bw}^{i-1}}^2_{\bK^{-1}}. \label{est5}
\end{align}
Thus, if the time step $\tau$ satisfies
\begin{equation}
0 < \tau < \frac{2(a_0 - b_0)}{c_f^2 M^2(\dfrac{k_M}{\theta_m} + 1) - \dfrac{\theta_m}{4c_{\Omega,d}}},
\end{equation}
we can write \eqref{est5} as
\begin{equation}
F^i \leq \frac{1}{1+\delta} F^{i-1}, 
\end{equation}
where
\begin{equation}
F^i := \frac{L_p}{2} \norm{e_p^i}^2 + \frac{L_T}{2} \norm{e_T^i} + \frac{\tau}{4} \norm{\be_{\bw}^i}^2_{\bK^{-1}},
\end{equation}
and
\begin{equation}
\delta := \min \bigg\{ \frac{2}{L_p}(c_0 - b_0), 
\frac{2}{L_T}\left(a_0 - b_0 + \frac{\tau \theta_m}{4c_{\Omega,d}} - \frac{\tau c_f^2M^2}{2}\left(\frac{k_M}{\theta_m} + 1\right)\right), \frac{1}{2} \bigg\} > 0.
\end{equation}
Going back to eq. \eqref{itere5}, we choose $\bv = \be_{\bu}^i$ as test function which leads to
\begin{align}
\nonumber 2\mu \norm{\bep(\be_{\bu}^i)}^2 + \lambda \norm{\nabla \cdot \be_{\bu}^i}^2
&= \alpha (e_p^i, \nabla \cdot \be_{\bu}^i) + \beta (e_T^i, \nabla \cdot \be_{\bu}^i) \\
&\leq (\alpha \norm{e_p^i} + \beta \norm{e_T^i})\norm{\nabla \cdot \be_{\bu}^i} \\
&\leq (\alpha \norm{e_p^i} + \beta \norm{e_T^i})
\left( \xi \sqrt{d} \norm{\bep(\be_{\bu}^i)} + (1-\xi)\norm{\nabla \cdot \be_{\bu}^i}\right),
\end{align}
for some $\xi \in (0,1)$. Following the same steps which led to \eqref{ubound}, and choosing as before $\xi = \dfrac{2\mu}{2\mu + d\lambda}$, we get by the Young inequality
\begin{equation}
\frac{\mu}{2}\norm{\bep(\be_{\bu}^i)}^2 + \frac{\lambda}{4} \norm{\nabla \cdot \be_{\bu}^i}^2
\leq \frac{2\alpha^2}{3(\frac{2\mu}{d} + \lambda)} \norm{e_p^i}^2
+ \frac{2\beta^2}{3(\frac{2\mu}{d} + \lambda)} \norm{e_T^i}^2.
\end{equation}
This shows a contraction of the residuals and therefore completes the proof.
\end{proof}
\begin{remark}[The other partially decoupled schemes]
 For the partially decoupled schemes~\textnormal{\textbf{HM-F}} 
 and~\textnormal{\textbf{FM-H}} (Algorithms~\ref{sect:algo3} and \ref{sect:algo4} respectively) the contractions can obtained similarly to 
 the scheme~\textnormal{\textbf{HF-M}} with minor changes in the coefficients.
\end{remark}
Before we state the last of our main results, we let the difference functions defined in~\eqref{diffdef2} now be the difference between the solutions at the iteration $i$ of problem~\eqref{decoupled2} and the solutions to~\eqref{discretevar}. The last of our main results then reads:
\begin{corollary}[Convergence of the fully decoupled algorithms]
Under the assumptions of Theorem~\ref{mainthm}, the fully decoupled $L$-scheme~\textnormal{\textbf{F-H-M}} (Algorithm~\ref{sect:algo6}) defines a contraction
\begin{align}
\nonumber&\left(a_0 - \frac{b_0}{2} + \frac{L_T}{2} + \frac{\tau\theta_m}{4c_{\Omega,d}} - \frac{\tau c_f^2M^2}{2} \left(\frac{k_M}{\theta_m} + 1\right)\right) \norm{e_T^i}^2 
 \\
\nonumber&\qquad\left(c_0 - b_0 + \frac{L_p}{2} \right) \norm{e_p^i}^2  
+ \frac{\tau}{2} \norm{\be_{\bw}^i}^2_{\bK^{-1}}+ \frac{\tau}{4} \norm{\be_{\br}^i}^2_{\bth^{-1}}\\
&\qquad \qquad \leq \left(\frac{L_T}{2} + \frac{b_0}{2}\right)\norm{e_T^{i-1}}^2\frac{L_p}{2}\norm{e_p^{i-1}}^2.\label{decoupled}
\end{align}
Furthermore, the estimate \eqref{est42} holds true.
\end{corollary}
\begin{proof}
We follow  the same lines as in the proof of Theorem~\ref{mainthm}, and take the difference of equations \eqref{decoupled2_1} -- \eqref{decoupled2_5} with the same equations solved by $(T^n, \br^n, p^n, \bw^n, \bu^n)$, and obtain the difference equations for the fully decoupled scheme~\textnormal{\textbf{F-H-M}}. We then promptly obtain estimate \eqref{decoupled}, from which the contraction is inferred by choosing the stabilization parameters and the time step. That of the second estimate follows in exactly the same way.
\end{proof}
\begin{remark}[The fully decoupled scheme~\textnormal{\textbf{H-F-M}}]
 The contraction~\ref{decoupled} holds true for the fully decoupled scheme~\textnormal{\textbf{H-F-M}} 
(Algorithm~\ref{sect:algo5}) by exchanging  in there the coefficients in the right-hand side, i.e.,  $\dfrac{L_p}{2}$ 
becomes $\dfrac{L_p}{2}+\dfrac{b_0}{2}$ and $\dfrac{L_T}{2}+\dfrac{b_0}{2}$ becomes $\dfrac{L_p}{2}$.
\end{remark}

\section{Numerical experiments}\label{sec:experiments}
In the following we present three numerical test cases using the algorithms from Section~\ref{sec:algos}. The first is a constructed problem, posed on the unit square domain, with prescribed solutions for the temperature, pressure and displacements. Here, we consider five different parameter regimes, exhausting all possibilities of weak/strong coupling between the subproblems, and compare the number of iterations needed for convergence with decreasing mesh sizes. Since analytical solutions are available, we present also discretization errors. 

Next, we present two implementations of Mandel's problem, which is originally a benchmark problem in linear poroelasticity, extended here to nonlinear thermo-poroelasticity. For the original Mandel problem, analytical solutions for the pressure and displacement field are known. Due to the similarity of the thermo-poroelastic equations we consider with the linear Biot's equations, and due to the lack of benchmark problems for thermo-poroelasticity, we choose to use this problem as our second and third numerical test cases. Even though the analytical solutions are no longer valid when including temperature, we have sufficiently weak temperature effects in the first implementation of Mandel's problem that the computed pressure and displacement field matches the (isothermal) analytical solutions. The second implementation of Mandel's problem includes a heat source, which has a significant effect on both the pressure and displacement. 
Regarding the spatial discretization, we choose the following finite element spaces:
\bse
\begin{align}
\Rcal_h, \Wcal_h & := \{ \psi \in H(\divr; \Omega) : \forall K \in \Xcal_h, \ \psi |_{K}  \in \rtzero(K)\},\\
\Tcal_h, \Pcal_h & := \{ \phi \in L^2(\Omega) : \forall K \in \Xcal_h, \ \phi |_K  \in \pspace_0(K) \},\\
\Ucal_h & := \{ \eta \in (H^1(\Omega))^d : \forall K \in \Xcal_h, \ \eta |_K \in [\pspace_1(K)]^d\},
\end{align}
\ese
where $\rtzero(K)$ denotes the
lowest-order Raviart--Thomas finite-dimensional subspace associated with the element $K\in\Xcal_h$, and $\pspace_{l}(K)$ is the space of polynomials on $K\in\Xcal_h$ of total degree less than or equal to $l$. Thus, the spaces  $(\Tcal_h,\Rcal_h)$ and $(\Pcal_h,\Wcal_h)$ 
are the  lowest order Raviart-Thomas  mixed finite element spaces for the mixed  flow and heat flow subproblems, respectively. Note that both spaces satisfy the condition \eqref{divcondition}, see \textit{e.g.}~\cite{gatica2014simple} for more details on (mixed) finite elements. The vector valued space $\Ucal_h$ is the  first order Lagrange finite element space for the mechanics problem.   We employ the following stopping criterion for the iterative algorithms, given in terms of the relative and absolute tolerances, $\textnormal{aTOL}$ and $\textnormal{rTOL}$, i.e. 
\begin{equation}\label{stop_crit}
\norm{(T^i, \br^i, p^i, \bw^i, \bu^i) - (T^{i-1}, \br^{i-1}, p^{i-1}, \bw^{i-1}, \bu^{i-1})} \leq \textnormal{aTOL} + \textnormal{rTOL}\norm{(T^i, \br^i, p^i, \bw^i, \bu^i)},
\end{equation} 
where we set $\textnormal{aTOL} = \textnormal{rTOL} =$ 1e-6 for all the computations. 
For the solution of the linear subproblems, we make use of  a direct sparse linear solver from the Python library SciPy, i.e., \texttt{scipy.sparse.linalg.spsolve}. The present  approaches can also be combined with  iterative solvers adapted to  the various subproblems. All numerical tests are implemented in a finite element code written in Python, the complete source code is accessible at \url{https://github.com/matkbrun/FEM}.  

\subsection{Test case~1: an academic example
with a manufactured solution}
As a first test case, we let the domain be a regular triangularization of the unit square, i.e., $\Omega = [0,1] \times [0,1] \subset \real^2$, and prescribe the following smooth solutions for the temperature, pressure and displacement
\bse
\begin{align}
T(x,t) &= t x_1(1-x_1)x_2(1-x_2), \\
p(x,t) &= t x_1(1-x_1)x_2(1-x_2), \\
\bu(x,t) &= t x_1(1-x_1)x_2(1-x_2) [1, 1]^\top,
\end{align}
where $x := (x_1, x_2) \in \real^2$, $t\geq 0$. The flux fields are then computed by
\begin{equation}
\br = - \bth \nabla T, \quad \text{ and } \quad \bw = - \bK \nabla p,
\end{equation}
\ese
while right hand sides, i.e., $z, g$ and $\bff$, can be calculated explicitly using equations \eqref{thporo:primalheat}--\eqref{thporo:primalmech}. We prescribe homogenous initial conditions and homogenous Dirichlet boundary conditions, for the temperature, pressure and displacement. All computations are done on a fixed time step, i.e., $\tau = 1.0$, and continued until criterion~\eqref{stop_crit} is satisfied.

For the analysis and comparison of our algorithms, we  consider  dimensionless equations, i.e. all parameters are set to $1.0e-1$, except for the three coupling coefficients $\{\alpha, \beta, b_0\}$, which we vary in order to \textit{weaken/strengthen} the coupling between the three subproblems. In particular, we consider five different parameter regimes, \textbf{PR1 -- PR5}, specified in Table~\ref{testcases}:

\begin{table}[h!]
\begin{center}
\begin{tabular}{| c || c | c | c | c | c | c |} 
\hline
		&\textbf{PR1}	&\textbf{PR2}	&\textbf{PR3}	&\textbf{PR4}	&\textbf{PR5}	\\ \hline \hline
$\alpha$	&1.0		&0.1  	&0.1  	&1.0		&0.1  	\\ \hline
$\beta$	&1.0		&0.1  	&1.0		&0.1  	&0.1  	\\ \hline
$b_0$	&1.0		&1.0		&0.1  	&0.1  	&0.1  	\\ \hline
\end{tabular}
\caption{Smooth solution: parameter regimes for varying strong/weak coupling between subproblems.}\label{testcases}
\end{center}
\end{table}
We also set $a_0 = c_0 = 2b_0$, thus satisfying~\ref{constraints}. Table~\ref{iterationscounts_theory} shows number of iterations needed for convergence using the six algorithms from Subsections~\ref{sect:algo1}, \ref{subsec:algo_two_levels} and~\ref{subsec:algo_three_levels}, for a single time step with decreasing mesh sizes, and stabilization according to equality in~\eqref{Lbound}. 
\begin{table}[h!]
\begin{center}
\begin{tabular}{| c || c | c | c | c | c || c | c | c | c | c |}
\hline
	&\textbf{PR1}	&\textbf{PR2}	&\textbf{PR3}	&\textbf{PR4}	&\textbf{PR5}	&\textbf{PR1}	&\textbf{PR2}	&\textbf{PR3} 	&\textbf{PR4} &\textbf{PR5} \\ \hline \hline
$h$ & \multicolumn{5}{c|| }{\textbf{HFM}} & \multicolumn{5}{c|}{\textbf{HF-M}} \\
\hline \hline
1/4 	&7	&3	&8	&8	&3	&31	&4	&11	&11	&4	\\
1/8 	&7	&3	&7	&7	&3	&35	&4	&13	&13	&4	\\
1/16	&6	&3	&7	&7	&3	&40	&4	&13	&13	&4	\\
1/32	&6	&3	&7	&7	&3	&41	&4	&13	&13	&4	\\
1/64	&6	&3	&7	&7	&3	&41	&4	&13	&13	&4	\\
\hline \hline
$h$ & \multicolumn{5}{c|| }{\textbf{HM-F}} & \multicolumn{5}{c|}{\textbf{FM-H}} \\
\hline \hline
1/4 	&9	&6	&8	&11	&4	&9	&6	&11	&8	&4	\\
1/8 	&9	&6	&7	&11	&4	&9	&6	&11	&7	&4	\\
1/16	&9	&6	&7	&11	&4	&9	&6	&11	&7	&4	\\
1/32	&9	&6	&7	&11	&4	&9	&6	&11	&7	&4	\\
1/64	&9	&6	&7	&11	&4	&9	&6	&11	&7	&4	\\
\hline \hline
$h$ & \multicolumn{5}{c|| }{\textbf{H-F-M}} & \multicolumn{5}{c|}{\textbf{F-H-M}} \\
\hline \hline
1/4 	&20	&6	&11	&11	&4	&20	&6	&11	&11	&4	\\
1/8 	&22	&6	&12	&12	&4	&22	&6	&12	&12	&4	\\
1/16	&24	&6	&13	&13	&4	&24	&6	&13	&13	&4	\\
1/32	&24	&6	&13	&13	&4	&24	&6	&13	&13	&4	\\
1/64	&24	&6	&13	&13	&4	&24	&6	&13	&13	&4	\\
\hline 
\end{tabular}
\caption{Smooth solution: number of iteration  with decreasing mesh sizes for parameter regimes \textbf{PR1 -- PR5}. Stabilization from theory.} \label{iterationscounts_theory}
\end{center}
\end{table}

We see that for parameter regimes $1$, $3$ and $4$ we have higher iterations numbers than for parameter regimes 2 and 5, for all six algorithms. This is because $L_T \sim \beta^2$ and $L_p \sim \alpha^2$, and larger stabilization results in higher iteration numbers. Furthermore, as expected, the strongly coupled parameter regime (\textbf{PR1}) yields the highest iteration numbers, in particular for the algorithms \textbf{HF-M}, \textbf{H-F-M} and \textbf{F-H-M}. Apart from this, the algorithms are performing robustly both with respect to different coupling regimes and decreasing mesh sizes. For comparison we also provide in Table~\ref{iterationscounts_nostab}, the results without  stabilization, i.e., $L_T = L_p = 0$.
\begin{table}[h!]
\begin{center}
\begin{tabular}{| c || c | c | c | c | c || c | c | c | c | c |}
\hline
	&\textbf{PR1}	&\textbf{PR2}	&\textbf{PR3}	&\textbf{PR4}	&\textbf{PR5}	&\textbf{PR1}	&\textbf{PR2}	&\textbf{PR3} 	&\textbf{PR4} &\textbf{PR5} \\ \hline \hline
$h$ & \multicolumn{5}{c|| }{\textbf{HFM}} & \multicolumn{5}{c|}{\textbf{HF-M}} \\
\hline \hline
1/4 	&3	&3	&3	&3	&3	&-	&4	&16	&16	&4	\\
1/8 	&3	&3	&3	&3	&3	&-	&4	&19	&19	&4	\\
1/16	&3	&3	&3	&3	&3	&-	&4	&20	&20	&4	\\
1/32	&3	&3	&3	&3	&3	&-	&4	&20	&20	&4	\\
1/64	&3	&3	&3	&3	&3	&-	&4	&20	&21	&4	\\
\hline \hline
$h$ & \multicolumn{5}{c|| }{\textbf{HM-F}} & \multicolumn{5}{c|}{\textbf{FM-H}} \\
\hline \hline
1/4 	&11	&6	&4	&22	&4	&11	&6	&21	&4	&4	\\
1/8 	&11	&6	&4	&23	&4	&11	&6	&23	&4	&4	\\
1/16	&12	&6	&4	&24	&4	&11	&6	&24	&4	&4	\\
1/32	&12	&6	&4	&24	&4	&12	&6	&24	&4	&4	\\
1/64	&12	&6	&4	&25	&4	&12	&6	&24	&4	&4	\\
\hline \hline
$h$ & \multicolumn{5}{c|| }{\textbf{H-F-M}} & \multicolumn{5}{c|}{\textbf{F-H-M}} \\
\hline \hline
1/4 	&34	&6	&17	&16	&4	&34	&6	&16	&17	&4	\\
1/8 	&38	&5	&19	&19	&4	&38	&5	&19	&19	&4	\\
1/16	&44	&5	&20	&20	&4	&44	&5	&20	&20	&4	\\
1/32	&46	&5	&20	&20	&4	&46	&5	&20	&21	&4	\\
1/64	&46	&5	&21	&20	&4	&46	&5	&20	&21	&4	\\
\hline 
\end{tabular}
\caption{Smooth solution: number of iterations  with decreasing mesh sizes for parameter regimes \textbf{PR1 -- PR5}. $L_T = L_p = 0$.} \label{iterationscounts_nostab}
\end{center}
\end{table}

We see here that the fully monolithic algorithm (\textbf{HFM}) has low iteration counts for all parameter regimes since this is only a linearization scheme, and does not require stabilization (cf. Theorem~\ref{thm:wellposedness}). For the two-level (Section~\ref{subsec:algo_two_levels}) and three-level (Section~\ref{subsec:algo_three_levels}) algorithms, which involves some splitting as well as linearization, we see that iteration counts for different parameter regimes corresponds to the various coupling/decoupling of the subproblems present in the algorithms (splitting of subproblems which are strongly coupled yields high iteration numbers, compared to solving the strongly coupled subproblems together). This is in contrast to employing stabilization, which greatly improves the robustness of the algorithms with respect to variations in parameters. For the strongly coupled parameter regime (\textbf{PR1}), we even have no convergence for algorithm \textbf{HF-M}, when no stabilization is applied. 

Furthermore, in order to check the robustness of the proposed schemes with respect to the nonlinearity, we adjust the coefficient of the nonlinear term, $c_f$, in order to make this term dominate. Table~\ref{iterationscounts_linearization} shows number of iterations needed for convergence when $c_f = 10$, for both the strongly coupled parameter regime (\textbf{PR1}) and the weakly coupled parameter regime (\textbf{PR5}). We also compare the results when no stabilization is applied. Note that we here only use a single mesh with $h = 1/16$.
\begin{table}[h!]
\begin{center}
\begin{tabular}{| c || c | c || c | c |}
\hline
Parameters	&\textbf{PR1}	&\textbf{PR5}	&\textbf{PR1}	&\textbf{PR5}	 \\ \hline \hline
\#	& \multicolumn{2}{c|| }{\textbf{HFM}} & \multicolumn{2}{c|}{\textbf{HF-M}} \\
\hline \hline
Non-stabilized 	&4	&4	&-	&5	\\ \hline
Stabilized 		&7	&4	&41	&5	\\
\hline \hline
\#	& \multicolumn{2}{c|| }{\textbf{HM-F}} & \multicolumn{2}{c|}{\textbf{FM-H}} \\
\hline \hline
Non-stabilized 	&11	&4	&10	&4	\\ \hline
Stabilized 		&9	&4	&8	&4	\\
\hline \hline
\#	& \multicolumn{2}{c|| }{\textbf{H-F-M}} & \multicolumn{2}{c|}{\textbf{F-H-M}} \\
\hline \hline
Non-stabilized 	&48	&5	&36	&4	\\ \hline
Stabilized 		&25	&5	&22	&4	\\
\hline 
\end{tabular}
\caption{Smooth solution: number of iterations with strong nonlinear effects, i.e. $c_f = 10$, and mesh size $h = 1/16$.} \label{iterationscounts_linearization}
\end{center}
\end{table}

For the weakly coupled parameter regime (\textnormal{\textbf{PR5}}), there is no difference in iteration numbers between the stabilized and non-stabilized algorithms, even with a dominating nonlinearity. For the strongly coupled parameter regime (\textnormal{\textbf{PR1}}), however, the stabilized algorithms has a significantly lower iteration count. This might be due to the fact that the nonlinearity appears as a coupling term.

Since analytical solutions are available for this problem, we provide also the discretization errors, denoted by $(e_{h,T}, e_{h,\br}, e_{h,p}, e_{h,\bw}, e_{h,\bu})$, measured in the $L^2$-norm. Due to almost no variation in discretization errors between the six algorithms and between the different parameter regimes (less than 5\%), we provide in Table~\ref{discretization_errors} the discretization errors using algorithm~\textbf{F-HM} applied on the weakly coupled parameter regime (\textbf{PR5}).
\begin{table}[h!]
\begin{center}
\begin{tabular}{| c || c | c | c | c | c |}
\hline
$h$	&$e_{h,T}$	&$e_{h,\br}$	&$e_{h,p}$	&$e_{h,\bw}$	&$e_{h,\bu}$	\\ \hline \hline
1/4	&8.5e-3		&3.5e-3		&8.5e-3		&3.5e-3		&5.6e-3		\\ 
1/8	&4.4e-3		&1.8e-3		&4.4e-3		&1.8e-3		&1.4e-3		\\ 
1/16	&2.2e-3		&9.3e-4		&2.2e-3		&9.3e-4		&3.6e-4		\\ 
1/32	&1.1e-3		&4.7e-4		&1.1e-3		&4.7e-4		&9.1e-5		\\ 
1/64	&5.5e-4		&2.3e-4		&5.5e-4		&2.3e-4		&2.3e-5		\\ 
\hline
\end{tabular}
\caption{Smooth solution: discretization errors  using algorithm~\textbf{F-HM} applied on the weakly coupled parameter regime (\textbf{PR5}), stabilization from theory.}\label{discretization_errors}
\end{center}
\end{table}
\subsection{Test case~2: Mandel's problem}
See~\cite{coussy2004poromechanics} for a detailed description of Mandel's problem. Formulas for the analytical pressure and displacements can be found in~\cite{phillips2008coupling}. We provide here only a brief description; Mandel's problem is posed on a rectangular domain representing a poroelastic slab of extent $2a$ in the horizontal direction, $2b$ in the vertical direction, and infinitely long in the third direction. The poroelastic slab in contained between two rigid plates, where at the initial time a downward force of magnitude $2F$ is applied to the top plate, with an equal but opposite force applied to the bottom plate. The top, left and bottom boundary is treated as impermeable, while zero pressure (and temperature) is prescribed at the right boundary. Due to the nature of Mandel's problem, the pressure, temperature and horizontal component of the displacement varies only in the horizontal direction, while the vertical component of the displacement varies only in the vertical direction. From symmetry considerations, it suffices to consider only the top right quarter rectangle, i.e. the computational domain is $[0,a] \times [0,b]$ (see Figure~\ref{fig:mandeldomain}).

\begin{figure}[h!]
  \centering
    \includegraphics[width=0.4\textwidth]{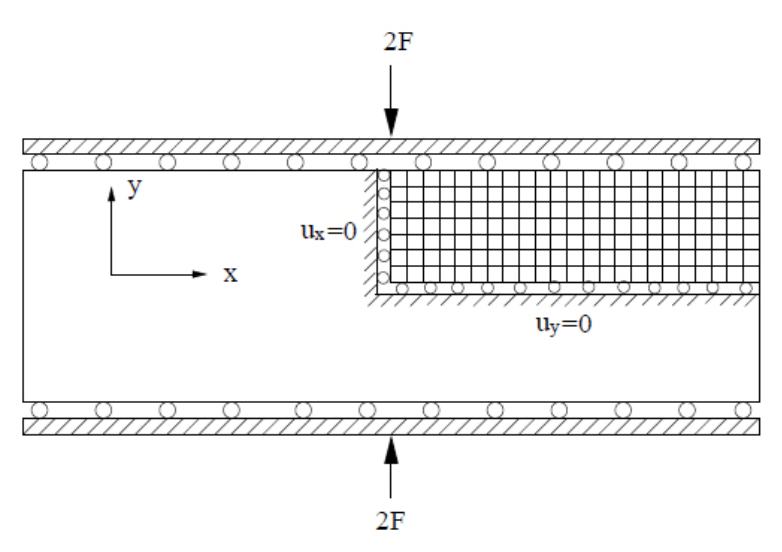}
  \caption{Setting of Mandel's problem quarter domain (figure from~\cite{mikelic2014numerical}).}\label{fig:mandeldomain}
\end{figure}

We perform now all computations with realistic choices of physical parameters. In particular, we take mechanics and flow parameters identical to~\cite{mikelic2014numerical}, and heat parameters identical to~\cite{kim2015unconditionally}. However, in~\cite{kim2015unconditionally} the flow-heat coupling coefficient $b_0$ is taken to be identically zero, so in order to preserve this coupling we instead choose a suitably small number (i.e. one that satisfies~\ref{constraints}). All parameters are listed in Table~\ref{mandel_physics}.

\begin{table}[h!]
\begin{center}
  \begin{tabular}{| l | l | l | l |}
    \hline
    Symbol & Quantity & Value & Unit \\ \hline \hline
    $E$ 	& Bulk modulus 	& 5.94e9 	& \SI{}{Pa} \\ 
    $\nu$ 	& Poisson's ratio 	& 0.2 	& - \\ 
    $c_0$  & Storage coefficient	& 6.06e-11 	& \SI{}{Pa^{-1}} \\
    $\alpha$  & Biot's coefficient 	& 1.0 	& -  \\
    $\mu_f$ & Fluid viscosity 	& 1.0e-3 	& \SI{}{Pa.s} \\
    $\hat{\bK}$ & Permeability & 9.87e-14 $\bid$ & \SI{}{m^2} \\
    $\bth$ & Effective thermal conductivity & 1.7 $\bid$ & \SI{}{W.m^{-1}.K^{-1}} \\
    $b_0$ & Thermal dilation coefficient & 3.03e-11 & \SI{}{K^{-1}} \\
    $\beta$ & Thermal stress coefficient & 9.9e6 & \SI{}{Pa.K^{-1}} \\
    $a_0$ & Effective heat capacity & 0.92e3 & \SI{}{J.kg^{-1}.K^{-1}} \\
    $T_{\textnormal{ref}}$ & Reference temperature & 298.15 & \SI{}{K} \\
    $c_f$ & Volumetric heat capacity fluid & 4.18e6 & \SI{}{J.m^{-3}.K^{-1}} \\
    $\tau$ & Time step & 10 & \SI{}{s} \\ 
    \hline
  \end{tabular} 
\end{center}
\caption{Mandel's problem: physical parameters, taken from~\cite{mikelic2014numerical,kim2015unconditionally}.}\label{mandel_physics}
\end{table}

In terms of our previous notation, we now have $\bK = \mu_f^{-1} \hat{\bK}$, and $\mu = \dfrac{E}{2(1+\nu)}$ and $\lambda = \dfrac{E\nu}{(1+\nu)(1+2\nu)}$. Note also that we will now employ the dimensional version of the heat equation~\eqref{thporo:primalheat}, which reads (in primal form)
\begin{equation}
\partial_{t}\left(a_0\frac{T}{T_{\textnormal{ref}}} - b_0 p + \beta \nabla \cdot \bu\right) + c_f(\bK \nabla p) \cdot \nabla \frac{T}{T_{\textnormal{ref}}} - \nabla \cdot \left(\bth \nabla \frac{T}{T_{\textnormal{ref}}} \right) =z.
\end{equation}
The magnitude of the compressive force is $F = \SI{2e8}{Pa.m}$, and the physical dimensions of the quarter rectangle is given by $a = \SI{100}{m}$ and $b = \SI{10}{m}$, of which we make a regular triangularization. We impose the compressive force as a Dirichlet boundary condition on the top boundary ($x_2 = b$) for the vertical component of the displacement. We denote by $n_{1}$ and $n_{2}$ the number of subdivisions of the domain in the $x_1$ and $x_2$ directions, respectively. For the first implementation of Mandel's problem we prescribe homogenous boundary conditions and source term for the heat problem. Figure~\ref{fig:plots_noheat} shows the solution profiles for the pressure, temperature and displacements for selected time steps, with the analytical (isothermal) solutions for the pressure and displacement included for comparison. 

\begin{figure}[h!]
    \centering
    \begin{subfigure}[b]{0.4\textwidth}
        \includegraphics[width=\textwidth]{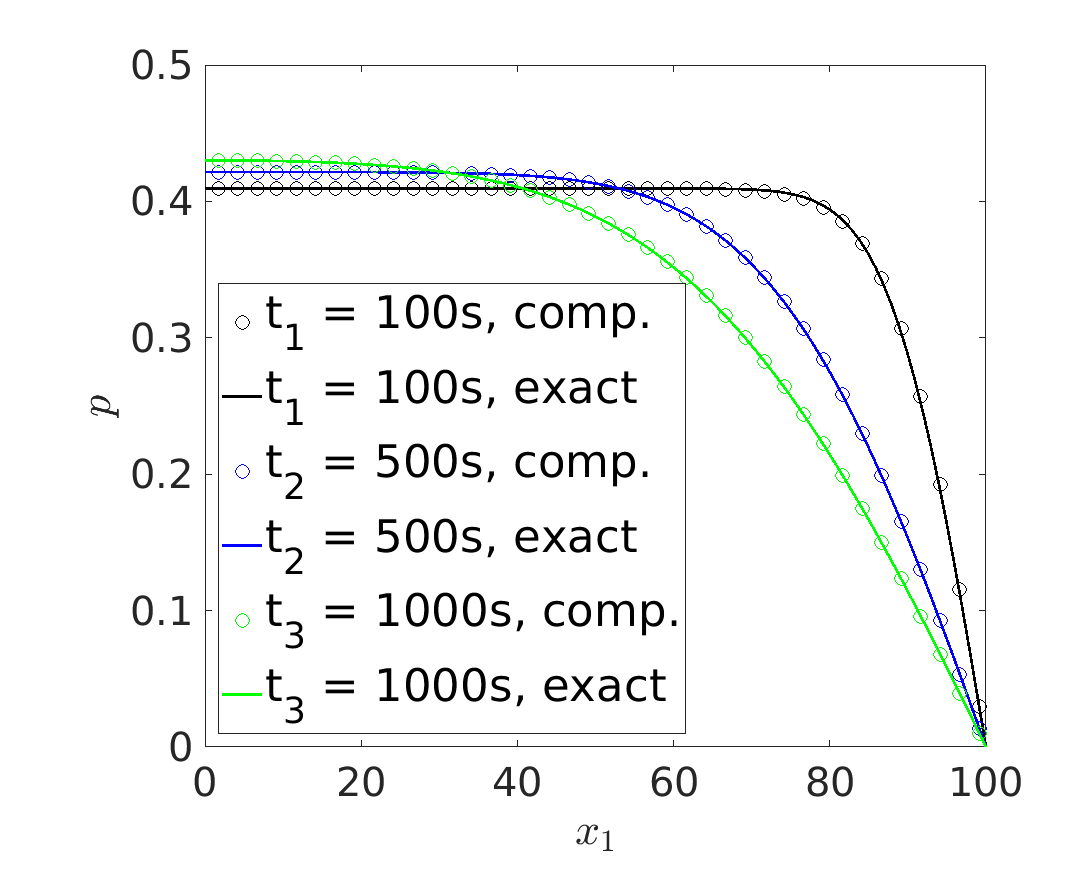}
        \caption{Pressure profile.}
        \label{fig:pressure_noheat}
    \end{subfigure}
    ~ \qquad
    \begin{subfigure}[b]{0.4\textwidth}
        \includegraphics[width=\textwidth]{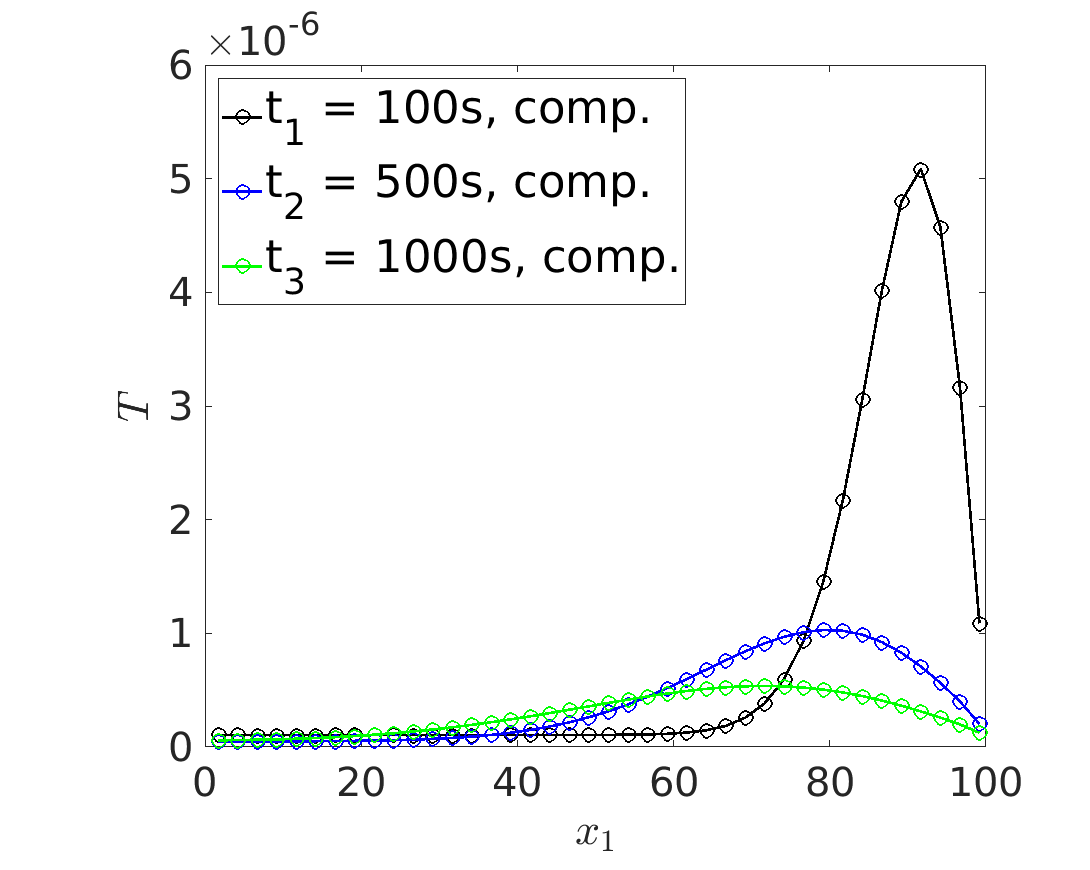}
        \caption{Temperature profile}
        \label{fig:temperature_noheat}
    \end{subfigure}
    
     \begin{subfigure}[b]{0.4\textwidth}
        \includegraphics[width=\textwidth]{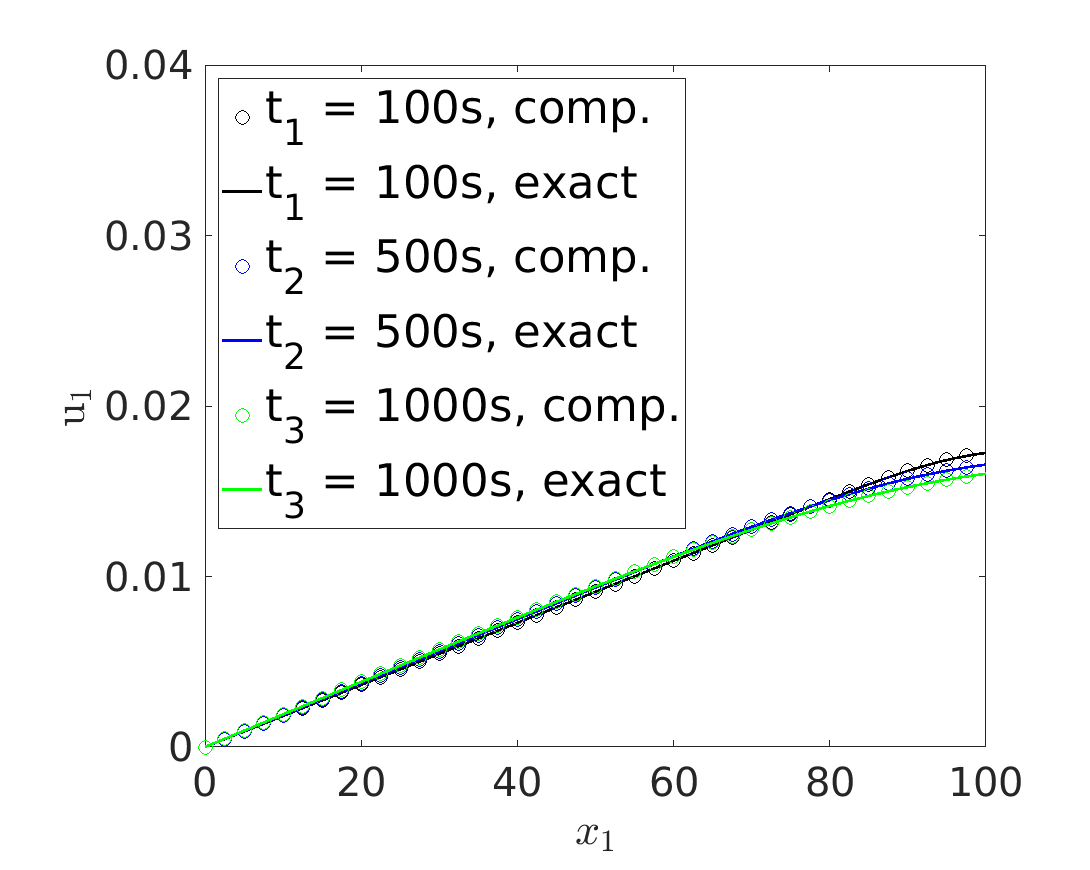}
        \caption{Displacement profile, 1st component}
        \label{fig:disp1_noheat}
    \end{subfigure}
    ~ \qquad
    \begin{subfigure}[b]{0.4\textwidth}
        \includegraphics[width=\textwidth]{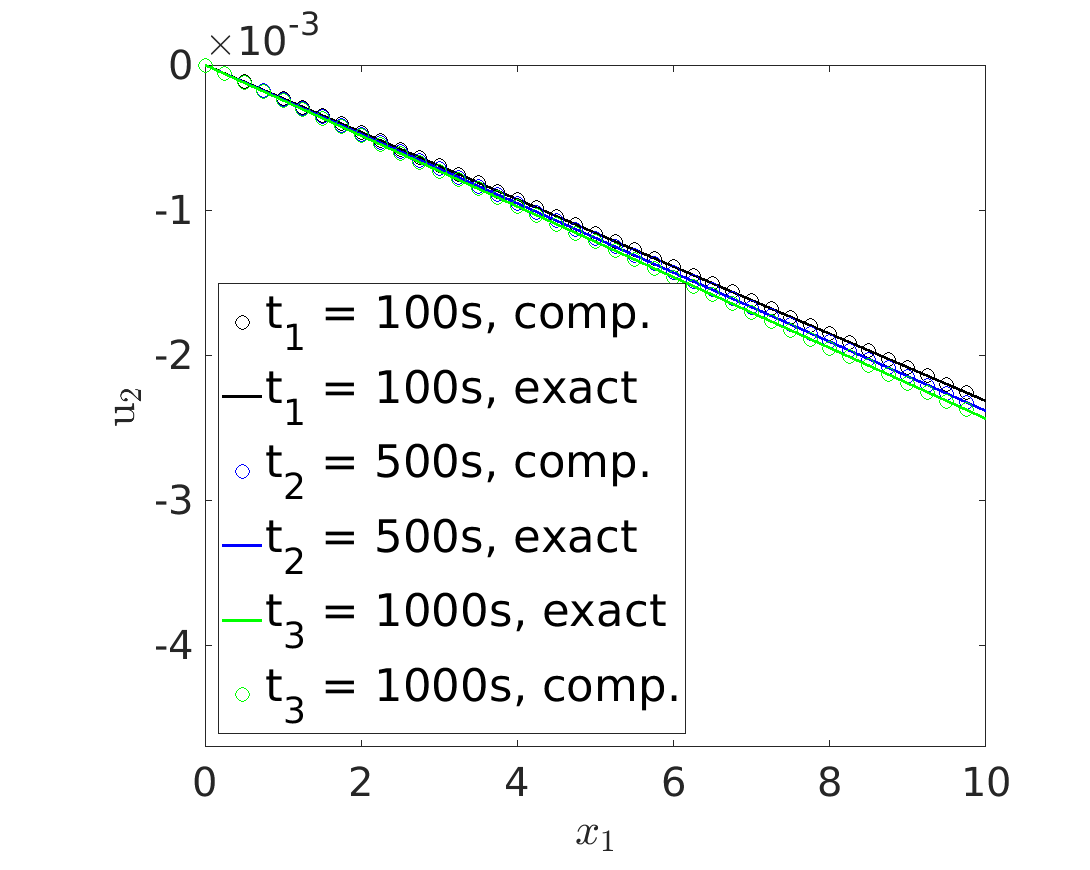}
        \caption{Displacement profile, 2nd component}
        \label{fig:disp2_noheat}
    \end{subfigure}
    \caption{Mandel's problem:  solution profiles for Mandel's problem at $t \in \{\SI{100}{s}, \SI{500}{s}, \SI{1000}{s}\}$, with $z = \SI{0}{W.m^{-3}.K^{-1}}$, and $n_1 = n_2 = 40$.}\label{fig:plots_noheat}
\end{figure}

The computed solutions for pressure and displacement matches the analytical ones, even though the analytical solutions are only valid for the linear isothermal problem. This is because the induced temperature effect in the system is small enough that the heat decouples from the flow and mechanics.
For the second implementation of Mandel's problem we prescribe a constant source term for the heat problem, i.e., $z = \SI{2e-4}{W.m^{-3}.K^{-1}}$. Figure~\ref{fig:plots} shows the solution profiles for the pressure, temperature and displacements at selected time steps. 

\begin{figure}[h!]
    \centering
    \begin{subfigure}[b]{0.4\textwidth}
        \includegraphics[width=\textwidth]{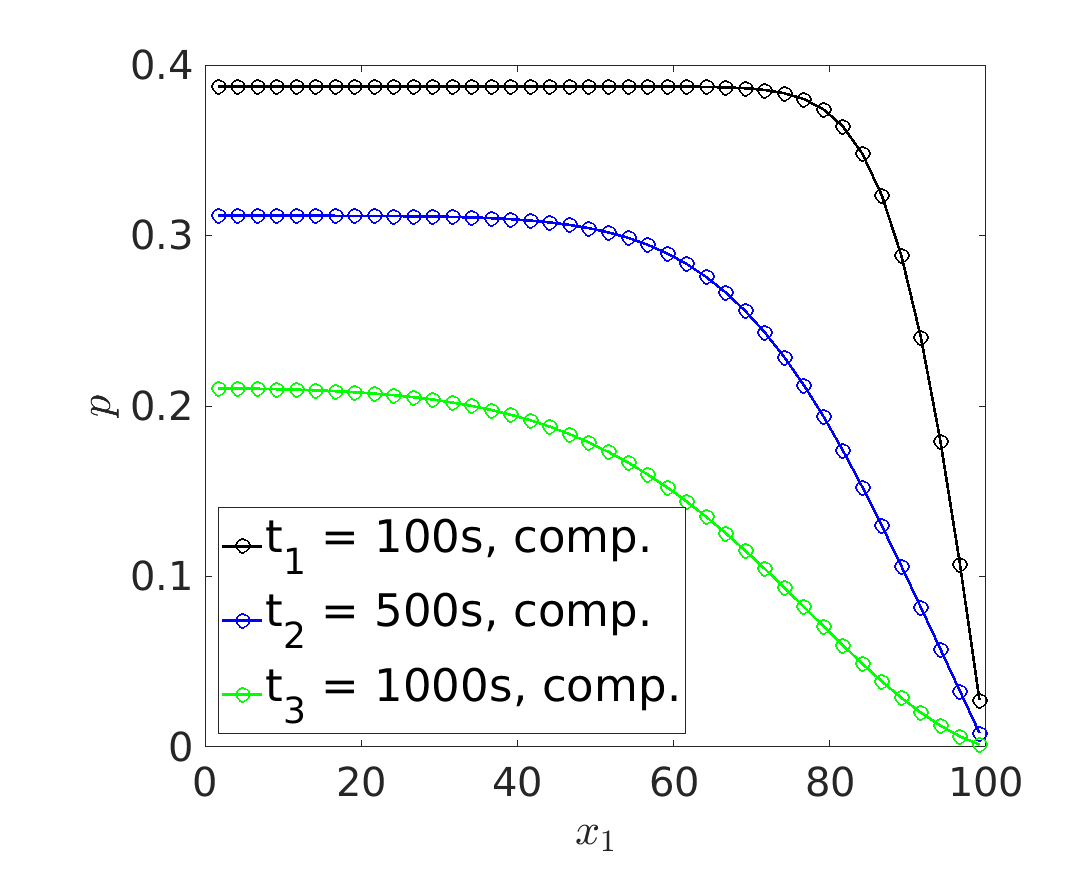}
        \caption{Pressure profile.}
        \label{fig:pressure}
    \end{subfigure}
    ~ \qquad
    \begin{subfigure}[b]{0.4\textwidth}
        \includegraphics[width=\textwidth]{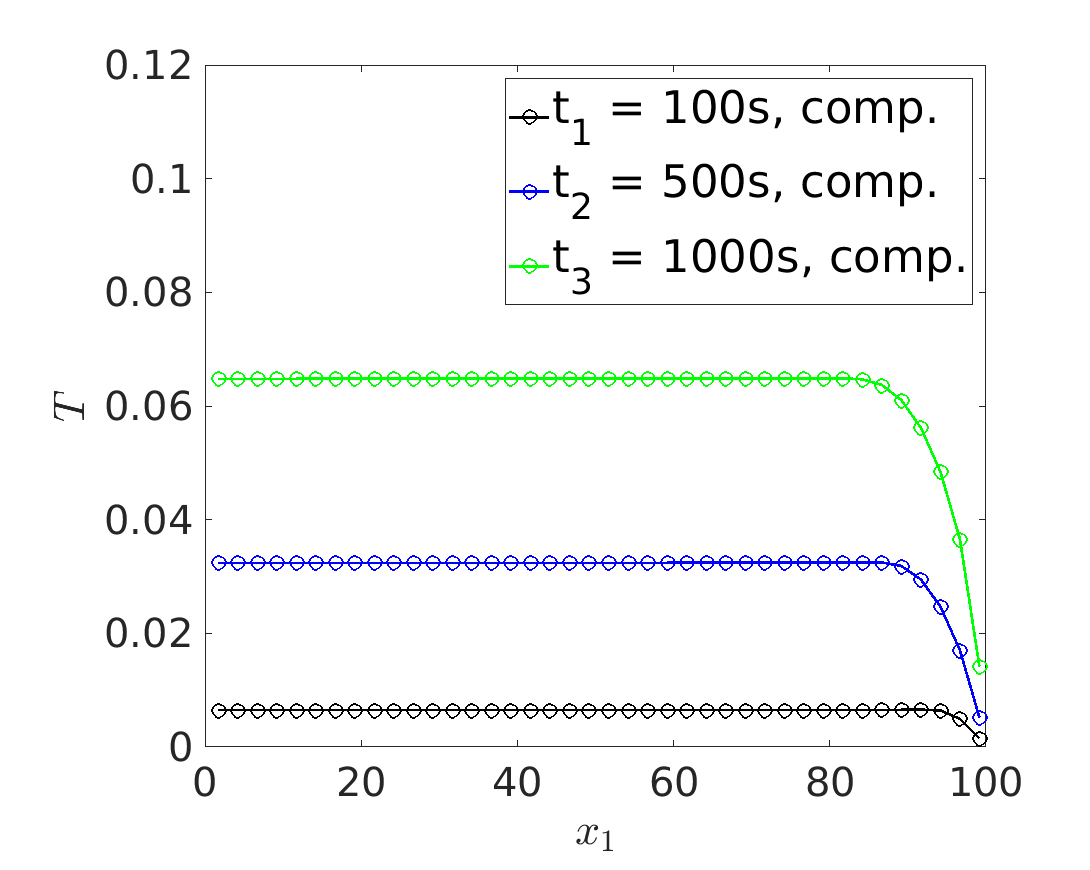}
        \caption{Temperature profile}
        \label{fig:temperature}
    \end{subfigure}
    
     \begin{subfigure}[b]{0.4\textwidth}
        \includegraphics[width=\textwidth]{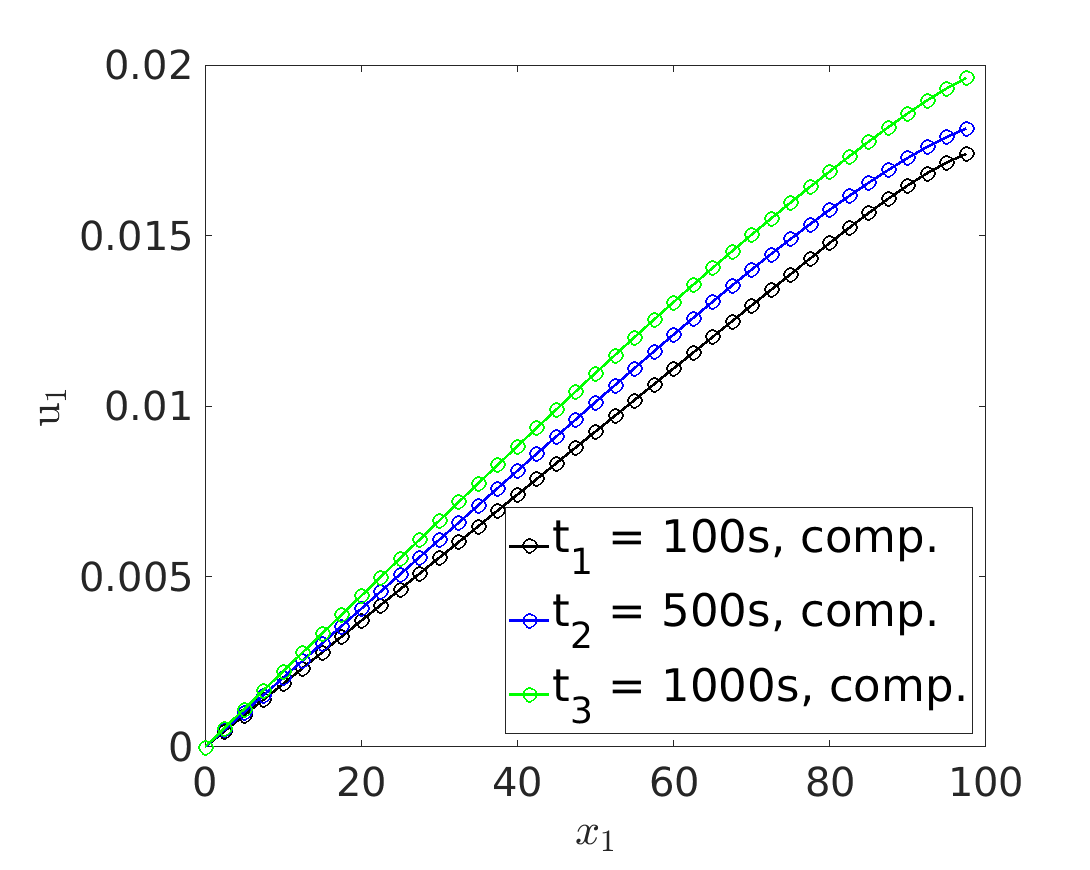}
        \caption{Displacement profile, 1st component}
        \label{fig:disp1}
    \end{subfigure}
    ~ \qquad
    \begin{subfigure}[b]{0.4\textwidth}
        \includegraphics[width=\textwidth]{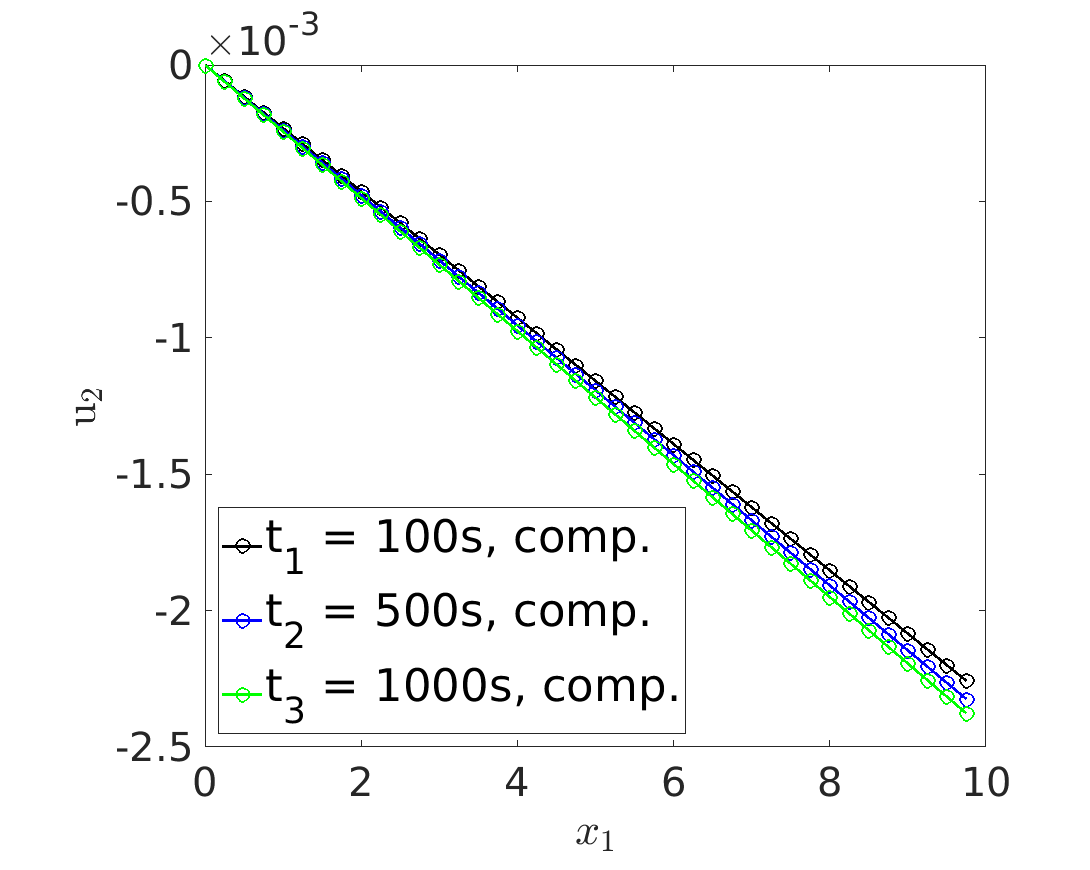}
        \caption{Displacement profile, 2nd component}
        \label{fig:disp2}
    \end{subfigure}
    \caption{Test case2: solution profiles for Mandel's problem at $t \in \{\SI{100}{s}, \SI{500}{s}, \SI{1000}{s}\}$, with $z = \SI{2e-4}{W.m^{-3}.K^{-1}}$, and $n_1 = n_2 = 40$.}\label{fig:plots}
\end{figure}

The temperature source now interacts with the other processes and thus has an effect on the pressure and horizontal component of the displacement. Furthermore, the temperature change in the system is now increasing with increasing time. Table~\ref{mandel:iterations} shows the number of iterations for Mandel's problem using the derived algorithms.
\begin{table}[h!]
\begin{center}
\begin{tabular}{| c || c | c || c | c || c | c |}
\hline
Heat source	&\ \ $z = 0$ \ \ &$z = $ 2e-4	&\ \ $z = 0$   \ \ &$z = $ 2e-4	&\ \ $z = 0$ 	\ \ &$z = $ 2e-4 \\ \hline \hline
$n_1 = n_2$	& \multicolumn{2}{c|| }{\textnormal{\textbf{HFM}} } & \multicolumn{2}{c||}{\textnormal{\textbf{HF--M}} } & \multicolumn{2}{c| }{\textnormal{\textbf{HM--F}} } \\
\hline \hline
10 	&18	&18	&14	&14	&14	&14			\\
20 	&18	&18	&13	&12	&13	&12			\\
40 	&18	&18	&13	&12	&13	&12			\\
\hline 
$n_1 = n_2$	& \multicolumn{2}{c|| }{\textnormal{\textbf{FM--H}} } & \multicolumn{2}{c||}{\textnormal{\textbf{H--F--M}} } & \multicolumn{2}{c| }{\textnormal{\textbf{F--H--M}} } \\
\hline \hline
10 	&18	&18	&14	&13	&14	&14			\\
20 	&18	&18	&13	&13	&13	&12			\\
40 	&18	&18	&13	&13	&13	&12			\\
\hline
\end{tabular}
\caption{Test case2: number of iterations  with decreasing mesh sizes for Mandel's problem. Stabilization from theory.}\label{mandel:iterations} 
\end{center}
\end{table}

\section{Conclusions}\label{sec:conclusions}
Based on developments on iterative splitting schemes from linear poroelasticity, we have proposed six novel iterative procedures for nonlinear thermo-poroelasticity. In particular, these algorithms are using stabilization and linearization techniques similar to~\cite{both2017robust, list2016study}, which is known in the literature as the `$L$-scheme'. The thermo-poroelastic problem we consider can be viewed as a coupling of three physical processes (or subproblems); flow, geomechanics and heat. Solving this system either monolithically (all three subproblems simultaneously), partially decoupled (two subproblems simultaneously), or fully decoupled (each subproblem separately), yields six possible combinations of coupling/decoupling which we have used to design the six algorithms. All of these involve a linearization of the convective term and added stabilization terms to both the flow and heat subproblems. In this sense, our use of the $L$-scheme is both as a stabilization for iterative splitting, and as a linearization of nonlinear problems. 

For any given situation the coupling strength between the three subproblems may vary. A-priori, the expectation is that solving together subproblems which are strongly coupled yields better efficiency properties than does splitting. On the other hand, if the coupling between two or more subproblems is weak, a splitting procedure might be beneficial. For this reason, and due to the fact that splitting the three-way coupled multi-physics problem into smaller subproblems, allows for combining existing codes that separately can handle any of the three processes involved (or two of them combined), six different algorithms are presented. These six algorithms covers all possibilities of strong/weak coupling between the three subproblems. Using the well-posedness of the continuous problem, we obtained lower bounds on the stabilization parameters, and proved the convergence of our proposed algorithms under a constraint on   the time step. In practice, however, we find that this bound is not tight; as long as the fluxes are not becoming unbounded (\textit{e.g.} due to a singularity) a `reasonable' time step can safely be chosen.

Our algorithms are tested in detail with several numerical examples. In particular, we find that all six algorithms are performing robustly with respect to both mesh refinement and different parameter regimes (i.e. strong/weak coupling between the subproblems and strong/weak nonlinear effects), using the stabilization revealed by our analysis. We also find that using no stabilization results in the algorithms being more sensitive to the parameter regimes, i.e. splitting subproblems which are strongly coupled yields high iteration numbers compared to solving these subproblems together. This phenomena is also observed in the stabilized algorithms, but to a significantly lesser extent. In particular, with no stabilization, each of the algorithms is suitable only for a certain parameter regime in contrast to the stabilized algorithms, which can handle a wide range of different parameters. 

\vspace*{-0.2cm}

\enlargethispage{0.2cm}

\bibliographystyle{siamplain}
\bibliography{ref}

\end{document}